\documentclass{article}
\usepackage{geometry}
\usepackage{graphicx}
\usepackage{amsmath}
\usepackage{amssymb}
\usepackage{color}
\usepackage{hyperref}
\usepackage{epsfig}
\usepackage{tikz}
\usepackage{pgfplots}
\usepackage{caption}
\usepackage{float}
\usepackage{subcaption}
\usepackage{enumitem}

\usepackage{bbm}
\usepackage{amsfonts}
\usepackage{amsmath, amsthm, amsfonts, amssymb, color}
\usepackage{mathrsfs}
\usepackage{color}

\allowdisplaybreaks

\newtheorem{remark}{Remark}

\newtheorem{thm}{Theorem}[section]
\newtheorem{cor}[thm]{Corollary}
\newtheorem{lem}[thm]{Lemma}
\newtheorem{prp}[thm]{Proposition}

\theoremstyle{definition}

\newcommand{\scr}[1]{\mathscr #1}
\definecolor{wco}{rgb}{0.5,0.2,0.3}

\numberwithin{equation}{section} \theoremstyle{remark}

\newcommand{\ua}{\uparrow}

\newcommand{\RR}{{\mathbb R}}

\newcommand{\notiz}[1]{}

\title{Milstein schemes and antithetic {\color{black}multilevel Monte Carlo} sampling for delay McKean--Vlasov equations and interacting particle systems}   
\author{Jianhai Bao\footnote{Center for Applied Mathematics, Tianjin University, 300072 Tianjin, China, E-mail: jianhaibao13@gmail.com. The author is supported by NNSFC (11801406).} , Christoph Reisinger\footnote{Mathematical Institute, University of Oxford, Andrew Wiles Building, Woodstock Road, Oxford, OX2 6GG, UK, E-mail: christoph.reisinger@maths.ox.ac.uk.
} , Panpan Ren\footnote{Mathematics Department, Bonn University, 53115, Germany, E-mail: rppzoe@gmail.com.} , Wolfgang Stockinger\footnote{Mathematical Institute, University of Oxford, Andrew Wiles Building, Woodstock Road, Oxford, OX2 6GG, UK, E-mail: wolfgang.stockinger@maths.ox.ac.uk. The author is supported by an Upper Austrian Government fund.}
}
\date{}
\begin{document}
\allowdisplaybreaks
\def\R{\mathbb R}  \def\ff{\frac} \def\ss{\sqrt} \def\B{\mathbf
B}
\def\N{\mathbb N} \def\kk{\kappa} \def\m{{\bf m}}
\def\ee{\varepsilon}\def\ddd{D^*}
\def\dd{\delta} \def\DD{\Delta} \def\vv{\varepsilon} \def\rr{\rho}
\def\<{\langle} \def\>{\rangle} \def\GG{\Gamma} \def\gg{\gamma}
  \def\nn{\nabla} \def\pp{\partial} \def\E{\mathbb E}
\def\d{\text{\rm{d}}} \def\bb{\beta} \def\aa{\alpha} \def\D{\scr D}
  \def\si{\sigma} \def\ess{\text{\rm{ess}}}
\def\beg{\begin} \def\beq{\begin{equation}}  \def\F{\scr F}
\def\Ric{\text{\rm{Ric}}} \def\Hess{\text{\rm{Hess}}}
\def\e{\text{\rm{e}}} \def\ua{\underline a} \def\OO{\Omega}  \def\oo{\omega}
 \def\tt{\tilde} \def\Ric{\text{\rm{Ric}}}
\def\cut{\text{\rm{cut}}} \def\P{\mathbb P} \def\ifn{I_n(f^{\bigotimes n})}
\def\C{\scr C}   \def\G{\scr G}   \def\aaa{\mathbf{r}}     \def\r{r}
\def\gap{\text{\rm{gap}}} \def\prr{\pi_{{\bf m},\varrho}}  \def\r{\mathbf r}
\def\Z{\mathbb Z} \def\vrr{\varrho} \def\ll{\lambda}
\def\L{\scr L}\def\Tt{\tt} \def\TT{\tt}\def\II{\mathbb I}
\def\i{{\rm in}}\def\Sect{{\rm Sect}}  \def\H{\mathbb H}
\def\M{\scr M}\def\Q{\mathbb Q} \def\texto{\text{o}} \def\LL{\Lambda}
\def\Rank{{\rm Rank}} \def\B{\scr B} \def\i{{\rm i}} \def\HR{\hat{\R}^d}
\def\to{\rightarrow}\def\l{\ell}\def\iint{\int}
\def\EE{\scr E}\def\no{\nonumber}
\def\A{\scr A}\def\V{\mathbb V}\def\osc{{\rm osc}}
\def\BB{\scr B}\def\Ent{{\rm Ent}}
\def\U{\scr U}\def\8{\infty} \def\si{\sigma}\def\1{\lesssim}

\maketitle 
\begin{abstract}
\noindent
In this paper, we first derive Milstein schemes for an interacting particle system associated with point delay McKean--Vlasov stochastic differential equations (McKean--Vlasov SDEs), possibly with a drift term exhibiting super-linear growth in the state component.
We prove strong convergence of order one and moment stability, making use of techniques from variational calculus on the space of probability measures with finite second order moments.
Then, we introduce an antithetic multilevel Milstein scheme, which leads to optimal complexity estimators for expected functionals of solutions to delay McKean--Vlasov equations without the need to simulate L{\'e}vy areas.
\end{abstract}

\section{Introduction} 
A McKean--Vlasov equation (introduced by McKean \cite{MK}) for a $d$-dimensional process $X = (X(t))_{t \in [0,T]}$, for some given $T>0$, is an SDE where the underlying coefficients depend on the current state $X(t)$ and, additionally, on the law of $X(t)$, i.e., 
\begin{equation}\label{McKeanLimit}
    \mathrm{d}X(t) =b(X(t), \mathcal{L}_{X(t)}) \, \mathrm{d}t + \sigma(X(t), \mathcal{L}_{X(t)}) \, \mathrm{d}W(t), \quad X_0 = \xi,
\end{equation}
where $W=(W(t))_{t \in [0,T]}$ is an $m$-dimensional standard Brownian motion, $ \mathcal{L}_{X(t)}$ denotes the marginal law of the process $X$ at time $t \in [0,T]$ and $\xi$ is an $\mathbb{R}^d$-valued random variable. The existence and uniqueness theory for strong solutions of McKean--Vlasov SDEs with coefficients of linear growth and Lipschitz type conditions (with respect to the state and the measure) is well-established (see, e.g., \cite{AS,MEL}). For further existence and uniqueness results on weak and strong solutions of McKean--Vlasov SDEs, we refer to \cite{RAY2,BBP,HSS,MVA,LACK} and references cited therein. 
It is also known from \cite{RST} that McKean--Vlasov SDEs with super-linear growth of the drift term with respect to the state variable
 admit a unique strong solution provided a one-sided Lipschitz condition holds.
Existence and uniqueness for path-dependent McKean--Vlasov SDEs (where the drift has super-linear growth) was studied in \cite{WANG}.

An illustrative example for path-dependent McKean--Vlasov SDEs, which will be the focus of this work, is the one-point delay McKean--Vlasov SDE 
\begin{equation*}
\mathrm{d}X(t) = b(X(t),X(t-\tau), \mathcal{L}_{X(t)},
\mathcal{L}_{X(t-\tau)}) \, \mathrm{d}t + \sigma(X(t),X(t-\tau),
\mathcal{L}_{X(t)}, \mathcal{L}_{X(t-\tau)}) \, \mathrm{d}W(t),
\end{equation*}
with $\tau >0$ a given {\color{black}deterministic delay parameter} and the initial datum $\xi$ is assumed to be a continuous ({\color{black}random}) function $\xi:\Omega \times [-\tau,0]\to\R^d$.

Such (McKean--Vlasov) SDEs with a time delay are relevant for applications, for instance, in medicine, biology, and finance, where the effect of some actions or causes is not immediately visible or significant. For example, infectious diseases have a certain incubation time, so that the currently observed and infectious cases were themselves infected at a prior time. Furthermore, neural networks can be modelled using delay mean-field equations 
to account for the delay in the transmission of signals amongst large numbers of neurons; 
see \cite{JONTO} for details. 
In finance, a collapse of a bank or an investment might not immediately have an impact on the entire system, but only after some time lag.
Delay McKean--Vlasov SDEs are also employed in computational finance to describe path-dependent volatility models and the associated calibration problems \cite{JGU, LBG}. The latter models are characterised solely by a diffusion coefficient (i.e., zero drift) and, compared to standard local (stochastic) volatility models, can capture prominent historical patterns of volatility.

In addition, several mean-field control problems described by stochastic delayed
differential equations of McKean--Vlasov type can be found in the literature; see, e.g., \cite{JPFZZ} and references cited therein.
Further background on the theory and applications of delay equations is found, e.g., in \cite{JKHVSM,VKAM}.
For examples of McKean--Vlasov SDEs with non-globally Lipschitz drift, without delay, we refer the reader to \cite{stock} and the references cited therein.


The simulation of McKean--Vlasov SDEs typically involves two steps: First, at each time $t$, the true measure $ \mathcal{L}_{X(t)}$ is approximated by the empirical measure 
\begin{equation*}
 \mu_t^{X,N}(\mathrm{d}x) := \frac{1}{N}\sum_{j=1}^{N} \delta_{X^{j,N}(t)}(\mathrm{d}x),
\end{equation*}    
where $\delta_{x}$ denotes the Dirac measure at point $x$ and, for $\mathbb{S}_N:= \lbrace 1, \ldots, N \rbrace$, $(X^{i,N})_{i \in \mathbb{S}_N}$  (a so-called 
\textit{interacting particle system}) is the solution to the $\mathbb{R}^{dN}$-dimensional SDE
\begin{equation}\label{IPSD}
\mathrm{d}X^{i,N}(t) = b(X^{i,N}(t),X^{i,N}(t-\tau), \mu_t^{X,N}, \mu_{t-\tau}^{X,N}) \, \mathrm{d}t + \sigma(X^{i,N}(t),X^{i,N}(t-\tau), \mu_t^{X,N}, \mu_{t-\tau}^{X,N}) \, \mathrm{d}W^{i}(t), 
\end{equation}
{\color{black}with $X_{0}^{i,N} = \xi^{i}$.
Here, $W^{i}$ and $\xi^{i}$, for $i \in \mathbb{S}_N$, are independent Brownian motions (also independent of $W$) and i.~i.~d.\ copies of $\xi$, respectively}.
We will be able to refer to \cite{CGR} for an approximation result, referred to as \textit{propagation of chaos}, in the delay setting.
In a next step, one typically needs to introduce a reasonable time-stepping method to discretise the particle system $(X^{i,N})_{i \in \mathbb{S}_N}$ over some finite time horizon $[0,T]$. This second step is the focus of the present work.


The literature on higher-order numerical schemes (i.e., beyond the basic Euler--Maruyama method of order 1/2) for classical delay SDEs is sparse and restricted to Lipschitz continuous coefficients. In \cite{HUMY}, the strong convergence of a Milstein scheme for point-delay SDEs 
was proven using an It\^{o} formula for so-called tame functions and techniques from anticipative calculus. 
For the analysis of a Milstein scheme without relying on methods from anticipative calculus, we refer to \cite{KPS}. 
To the best of our knowledge, there is no published Milstein scheme for classical point-delay SDEs where the coefficients have super-linear growth.

It is also well-documented that standard schemes, e.g., of Euler-type with uniform step-size, are not suitable either for standard SDEs (see \cite{HJK2}) or (non-delay) McKean--Vlasov SDEs
(see \cite{RES, stock}) with super-linear growth in the drift, i.e., the moments of the discretised process explode as the mesh-size tends to zero (even though a unique solution with bounded moments {\color{black}to the non-discretised process exists}).
For remedies in these cases using adaptive, truncated, {\color{black}implicit} or tamed schemes see the references given in the above mentioned works.

For tamed Euler--Maruyama schemes applied to certain delay SDEs, we refer to \cite{KUMARSA} and the references cited therein. An Euler--Maruyama type scheme for the approximation of delay McKean--Vlasov SDEs was introduced in \cite{RENWU}, but
without an implementable (particle) approximation of the true measure. We are not aware of any published higher-order time-stepping schemes in this context, not even for globally Lipschitz continuous coefficients.
In this work, we provide a strong convergence analysis of first order time-stepping schemes for delay McKean--Vlasov equations, possibly with super-linear drift.

A first contribution of the current article is the following:
\begin{itemize}
\item
We derive a Milstein scheme for point delay McKean--Vlasov SDEs, i.e., path-dependent equations where the coefficients depend on the values of the process at a finite number of time points,
and its law at these points (but not on the entire path). 
\item
{\color{black}
We prove first order strong convergence for such problems,
where the drift is allowed to grow super-linearly in the current state but does not depend on delayed values (Theorem \ref{THDelay:THDelay4} for a `tamed' scheme). We will also require Lipschitz dependence of all coefficients on the measure.} \footnote{We refer the reader to Remark \ref{rem:anticip} for a discussion on technical difficulties which prevented us from considering delayed components in the drift term for the super-linear case.}
\end{itemize}
We hereby extend ideas from \cite{stock,K} on tamed Milstein schemes for non-delay equations.
To derive our main result, we make use of the Lions derivative of functionals acting on $\mathcal{P}_2(\RR^d)$.
As we do not employ an It\^{o} formula for the analysis of the scheme, we only need to require the coefficients to be once continuously differentiable in all variables (state and measure component).

The second, and main, novelty of the article concerns efficient estimators of expected functionals of the solution to
McKean--Vlasov SDEs with delay, on the basis of the aforementioned Milstein schemes.
The crucial computational difficulty for such equations is that already in the one-dimensional case the simulation of iterated stochastic integrals (see \cite{KP}) is required,
as a double stochastic integral of an earlier segment of the Wiener path with respect to the `current' Wiener path appears in the delay setting. Hence, for a direct application of the scheme, appropriate approximation techniques for such integrals (e.g., as proposed in \cite{WSON}) need to be used and add significant computational complexity.
Therefore, for the estimation of expected (path) functionals, we propose an antithetic approach that avoids the simulation of these iterated integrals, but still allows us to obtain higher order convergence of the 
correction terms within a multilevel Monte Carlo estimator (MLMC; see, \cite{MG,MB2}).
Consequently, we recover optimal complexity of the MLMC sampling scheme.
Here, we adapt ideas from \cite{MGLS} for classical multidimensional SDEs
without delay or mean-field interaction. 
Our results are new also for delay equations without mean-field interaction.

A natural MLMC estimator for $\mathbb{E}[P(X(T))]$, where $P$ is a regular enough function (see Corollary \ref{cor2}) and $X$ is the solution to the (delay) McKean--Vlasov SDE, 
can be constructed as  follows (see \cite{AT}):

First, for a given integer $L$ and $l \in \{0,1,\ldots, L \}$, we  consider  numerical  
approximations $(Y^{j,N,l}(T))_{j \in \mathbb{S}_N}$ to the particle system
at time $T$, obtained by a time-stepping scheme with mesh-size $\delta_l>0$.
Then, we introduce the quantity  
\begin{equation}
\label{level_l}
\phi_N^{l}:= \frac{1}{N} \sum_{j=1}^{N} P\left( Y^{j,N,l}(T) \right),
\end{equation}  
on the time discretisation level $l$.
In the sequel, we fix the number of particles $N$ across all levels and construct independent realisations of $\phi_N^{l}$. 

The set of random variables needed to generate the $k$-th realisation of $\phi_N^{l}$ is denoted by $\boldsymbol{\omega}^{(l,k)}_{1:N}$, $k \in \lbrace 1, \ldots, K_l \rbrace$ for some given integer $K_l$; this is the set of $N$ independent copies $(W^{i}, \xi^{i})_{i \in \mathbb{S}_N}$. 
For $L \in \mathbb{N}$ levels, a MLMC estimator for $\mathbb{E}[P(X(T))]$ can be written as
\begin{equation}
\label{eqn:MLML}
\sum_{l=0}^{L} \frac{1}{K_l} \sum_{k=1}^{K_l} \left(\phi_{N}^{l} - \varphi_{N}^{l-1} \right) \left(  \boldsymbol{\omega}^{(l,k)}_{1:N}\right),
\end{equation}    
where $\varphi_{N}^{l-1}$ has to be chosen appropriately. In particular, if
$\varphi^{{-1}}_{N}=0$ and $\mathbb{E}[\phi_{N}^{l-1} ] = \mathbb{E}[\varphi_{N}^{l-1}]$, then a telescoping property applies, which ensures the bias of the MLMC estimator
for given $L$ is the same as that of the standard MC estimator $\phi_N^{L}$ with mesh-size $\delta_L$.
The key to reducing the computational complexity of the MLMC estimator is to introduce a strong correlation between $\varphi$ and $\phi$ to reduce the variance.
A simple choice to achieve this is 
{\color{black}$\varphi_{N}^{l-1} (\boldsymbol{\omega}^{(l,k)}_{1:N}) = \phi_{N}^{l-1}(\boldsymbol{\omega}^{(l,k)}_{1:N})$},
i.e., the same Brownian motions and initial data are used for the two particle systems (with the same $N$) on the fine grid and on the coarse grid.

In the present paper,
\begin{itemize}
\item
we show (in Theorem \ref{TheoremAntithetic}) that a certain antithetic Milstein scheme for $\phi_{N}^{l}$ gives strong order one in $\delta_l$ for the multilevel corrections
$\phi_{N}^{l} - \varphi_{N}^{l-1}$, without the need to compute L{\'e}vy areas,
\end{itemize}
 and hence at the same cost per path as the lower-order Euler--Maruyama scheme.
For simplicity, we restrict the analysis to the case where the drift coefficient depends only on the state and the measure, and the diffusion coefficient only on the state
and a single delayed value, but the results extend straightforwardly to measure dependent volatility. 
In addition, it is assumed $d=m=1$ for notational convenience.

Optimal complexity results for multilevel simulation of interacting particle systems are found in \cite{AT}, which are based on the assumption of
a multiplicative bound for the variance of the correction terms, 
\begin{align}\label{eq:VarAssump}
\mathbb{V} \left[ \phi_{N}^{l} - \varphi_{N}^{l-1} \right] \1 \frac{\delta_l^s}{N}.
\end{align}
Here, $s/2$ corresponds to the strong order of the time-stepping scheme.
Although the assumption (\ref{eq:VarAssump}) seems reasonable, it is an open problem to prove such multiplicative estimates (in $N$ and $\delta_l^s$).
The analysis in \cite{AT} subsequently reveals that higher order time-stepping schemes (e.g., $s=2$) can reduce the overall computational complexity.

For the following discussion, we assume that 
\eqref{eq:VarAssump} holds with $s=2$, i.e., in addition to the first order convergence in $\delta_l$, which we prove, we assume the factor $1/N$.
We also assume that
the complexity of computing all interaction terms for $N$ particles due to the presence of the empirical measure is $\mathcal{O}(N^{p+1})$ per time-step, where $p=0,1$.
To be more explicit, for $p=0$, the cost is of identical order to standard systems of SDEs, which is the case, for instance, if the coefficients involve
the expectation of $X(t)$, which can be approximated by the empirical average over all particles in a single computation.
The case $p=1$ would be relevant, e.g., if different expectations needed to be computed for each particle.

Then, \cite[Theorem 3.1]{AT} allows us to conclude that the work complexity required to compute $\mathbb{E}[P(X(T))]$ with RMSE of order $\vv>0$ is $\mathcal{O}(\vv^{-2-p})$ when our antithetic Milstein scheme is used as a time-stepping method. For a tamed Euler--Maruyama scheme, the multilevel scheme outlined above would give a work complexity of order $\mathcal{O}\left(\vv^{-2-p} \log^2(\vv) \right)$.
A standard MC estimator based on $K=\mathcal{O}\left(\vv^{-1} \right)$ samples of the entire particle system would have complexity $\mathcal{O}(\vv^{-3-p})$; see \cite{AT,LSAT}.

As noted in \cite{AT} and further analysed in \cite{LSAT}, one can additionally vary the number of particles with the levels (see also earlier \cite{BHR} for the case with interaction  through common noise). If 
we want to work with two particle systems of size $N_l$ and $N_{l-1}$ on level $l$, where $N_l = \beta N_{l-1}$ for some integer $\beta$ (e.g., $\beta=2$), we split the set of $N_l$ underlying Brownian motions into $\beta$ sets of Brownian motions and simulate independently $\beta$ particle systems each of size $N_{l-1}$. Hence, we can define
\begin{equation*}
{\color{black}\varphi^{l-1}_{N_{l-1}}(\boldsymbol{\omega}_{1:N_l}^{(l,k)}): = \frac{1}{\beta} \sum_{i=1}^{\beta} \phi^{l-1}_{N_{l-1}}(\boldsymbol{\omega}^{(l,k)}_{(i-1)N_{l-1}+1:iN_{l-1})}).}
\end{equation*} 

Using this definition in conjunction with \eqref{eqn:MLML} gives then a MLMC estimator for $\mathbb{E}[P(X(T))]$ where both the number of particles and number of time-steps vary with the level.
In \cite[Theorem 6.2]{LSAT},
for the setting without delay and with a constant diffusion coefficient (such that a standard Euler--Maruyama scheme already yields strong convergence of order one; see \cite[Lemma 6.1]{LSAT}), an additive error bound for the variance of the correction terms, of order one with respect to both $M$ (number of time-steps) and $N$, was proven.

Our main result on the strong convergence of the multilevel correction terms using the antithetic Milstein scheme for the delay case complements \cite[Theorem 6.2]{LSAT}
in the sense that we prove order two convergence of the variance of the correction terms in the setting with non-constant diffusion term without increasing the computational complexity compared to a standard (tamed) Euler--Maruyama scheme. We achieve this through an antithetic approach by avoiding the simulation of iterated stochastic integrals required for the direct Milstein scheme even for $d=m=1$ in the delay case.

The analysis in \cite[Theorem 6.3]{LSAT} then reveals that the complexity of such a multilevel scheme
is at most $\mathcal{O}(\vv^{-2-p})$ for a RMSE of order $\vv$. This coincides with the result of the earlier discussion, where the number of particles was fixed across all time discretisation levels. However, this holds now also without the assumption of a multiplicative error bound as in \eqref{eq:VarAssump}. This means that if both, the number of time-steps and number of particles, are varied with the level, then an additive bound for the variance decay yields the same computational complexity as a MLMC scheme with a multiplicative rate (\ref{eq:VarAssump}) but where only the number of time-steps are varied across the levels. Also, according to Table 1 in \cite{AT}, if $s=2$ and the computation of the interaction terms is of order $N^{2}$, then both MLMC estimators, i.e., where either only the number of time-steps or where both the number of particles and the number of time-steps vary with the levels, yield the same computational complexity. However, if the complexity in terms of $N$ is only of order one and $s=2$, then Table 1 in \cite{AT} reveals that it is more efficient to keep the number of particles constant across all levels; see Section \ref{SEC:TR} for details.  

The remainder of this article is organised as follows: In Section \ref{Section:Sec2}, we introduce the precise class of delay equations considered here and state the first order convergence result for the {\color{black}tamed} Milstein schemes. 
Section \ref{Section:SectionAnti} is concerned with the analysis of the antithetic approach. {\color{black}This allows us to construct a numerical scheme without simulating the iterated stochastic integrals, which still achieves convergence order one in a MLMC framework.}
Section \ref{Section:Sec4} illustrates the numerical performance of the proposed time-stepping schemes. The proofs of the main convergence results are deferred to Sections \ref{Section:Sec3} and \ref{panti}.  \\ \\

\noindent
\textbf{Preliminaries:} \\ \\
\noindent
We end this section by introducing some notations and concepts that will be needed throughout this article.
Let $T >0$ be a given terminal time and $(\Omega,\mathcal{F},(\mathcal{F}_t)_{t \in [0,T]},\mathbb{P})$ will denote a complete filtered probability space satisfying the usual assumptions, where $(\Omega,\mathcal{F},\mathbb{P})$ is assumed to be atomless.

Here, $(\R^d,\<\cdot,\cdot\>,|\cdot|)$ will represent the $d$-dimensional
Euclidean space and $\R^d\otimes\R^m$ be the collection of all $d\times m$-matrices. 

In addition, we use $\mathcal{P}(\R^d)$ to denote the family of all probability
measures on $\R^d$ and define the subset of probability measures with finite second order moment by
\begin{equation*}
\mathcal{P}_2(\R^d):= \Big \{ \mu\in\mathcal
{P}(\R^d): \ \ \mu(|\cdot|^2)= \int_{\R^d}|x|^2\mu(\d x)<\8 \Big \}.
\end{equation*}
For matrices, we will use the standard Hilbert-Schmidt norm denoted by $\| \cdot \|$.  

As metric on the space $\mathcal{P}_2(\R^d)$, we use the Wasserstein distance. For $\mu,\nu\in\mathcal{P}_2(\R^d)$, the Wasserstein distance between $\mu$ and $\nu$ is defined as
\begin{equation*}
\mathcal{W}_2(\mu,\nu) := \inf_{\pi\in\mathcal
{C}(\mu,\nu)} \left( \int_{\R^d\times \R^d}|x-y|^2\pi(\d x,\d y) \right)^{1/2},
\end{equation*}
where $\mathcal {C}(\mu,\nu)$ is the set of all couplings of $\mu$ and
$\nu$, i.e., $\pi\in\mathcal {C}(\mu,\nu)$ if and only if
$\pi(\cdot,\mathbb{R}^d)=\mu(\cdot)$ and $\pi(\mathbb{R}^d,\cdot)=\nu(\cdot)$. 

We briefly introduce the $L$-derivative of a functional $f: \mathcal{P}_2(\mathbb{R}^d) \to \mathbb{R}$, as it will appear in the proofs presented in the main section. For further information on this concept, we refer to \cite{PLI} or \cite{BLPR,HSS}. Here, we follow the exposition of \cite{CD}. We will associate to the function $f$ a lifted function $\tilde{f}$ by $\tilde{f}(X)=f(\mathcal{L}_X)$, where $\mathcal{L}_X$ is the law of $X$, for $X \in L_2(\Omega, \mathcal{F},\mathbb{P};\mathbb{R}^d)$.

A function $f$ on $\mathcal{P}_2(\mathbb{R}^d)$ is said to be $L$-differentiable at $\mu_0 \in \mathcal{P}_2(\mathbb{R}^d)$ if there exists a random variable $X_0$ with law $\mu_0$, such that the lifted function $\tilde{f}$ is Fr\'{e}chet differentiable at $X_0$. 

Now, the Riesz representation theorem implies that there is a ($\mathbb{P}$-a.s.) unique $\Phi \in L_2(\Omega, \mathcal{F},\mathbb{P};\mathbb{R}^d)$ with
\begin{equation*}
\tilde{f}(X) = \tilde{f}(X_0) + \langle \Phi,  X-X_0 \rangle_{L_2}  + o(\| X-X_0\|_{L_2}), \text{ as } \| X-X_0\|_{L_2} \to 0,
\end{equation*}
with the standard inner product and norm on $L_2(\Omega, \mathcal{F},\mathbb{P};\mathbb{R}^d)$. If $f$ is $L$-differentiable for all $\mu_0 \in \mathcal{P}_2(\mathbb{R}^d)$, then we say that $f$ is $L$-differentiable.
 
It is known (see, e.g., \cite[Proposition 5.25]{CD}) that there exists a Borel measurable function ${\color{black}\kappa}: \mathbb{R}^d \to \mathbb{R}^d$, such that $\Phi = {\color{black}\kappa}(X_0)$ almost surely, and hence
\begin{equation*}
f(\mathcal{L}_X) = f(\mathcal{L}_{X_0}) + \mathbb{E}\left\langle \kappa(X_0), X -X_0 \right \rangle +o(\| X-X_0\|_{L_2}).
\end{equation*} 
{\color{black}Note that $\Phi$ can be expressed as a function ($\kappa$) of any random variable $X_0$ with distribution $\mu_0$, giving a meaning to the $L$-derivative independently of the lifting chosen to construct it, see \cite[Remark 5.26]{CD}.}  
We define $\partial_{\mu}f(\mathcal{L}_{X_0})(y):={\color{black}\kappa}(y)$, $y \in \mathbb{R}^d$, as the $L$-derivative of $f$ at $\mu_0$. For a vector-valued (or matrix-valued) function $f$, these definitions have to be understood componentwise.

For some real $\tau>0$, let $\C=C([-\tau,0];\R^d)$ be the set of all
continuous functions $f:[-\tau,0]\to\R^d$. Further, for $\eta \in\C$ and $\theta\in[-\tau,0]$, let
$\Pi_{\theta}:\C\to\R^d$ be the projection operator, i.e., $\Pi_\theta(\eta)=\eta(\theta)$. For
$f\in C([-\tau,T];\R^d)$ and for $t \in [0,T]$, let $f_t\in\C$ be defined by
$f_t(\theta)=f(t+\theta),\theta\in[-\tau,0].$ In the literature (see e.g., \cite{XM}),
$(f_t)_{t\ge0}$ is called the segment process associated with
$(f(t))_{t  \geq -\tau}$. We will also need the norms {\color{black} $\| f \|_{\infty}:=\sup_{v \in [-\tau,0]} |f(v)|$} for $f \in \C$ and $\|X\|_{\8,t}:= \sup_{s \in [0,t]} |X(s)|$, for $t \geq 0$ ({\color{black}and zero otherwise}) and a process $X$ precisely defined below. 
\section{Delay McKean--Vlasov SDEs and Milstein schemes}\label{Section:Sec2}

\subsection{Problem formulation}
In this section, we consider the following SDE  for $t \in [0,T]$
\begin{equation}\label{E00}
\d X(t)=b(\Pi(X_t),\mathcal{L}_{\Pi(X_t)}) \, \d
t+\si(\Pi(X_t),\mathcal{L}_{\Pi(X_t)}) \, \d W(t), \quad X_0=\xi\in L^{0}_p(\C),
\end{equation}
{\color{black}where $L^{0}_p(\C)$ is the space $\mathcal{F}_0$-measurable $\C$-valued random variables with $\mathbb{E} [\| \xi \|_{\infty}^p] < \infty $, for a given $p \geq 2$. Further, for {\color{black} $s_1 > s_2 > \ldots > s_k \in[-\tau,0]$}, {\color{black}with $\tau <T$} and a $\C$-valued random
variable $\chi$}
$$\Pi(\chi):=(\Pi_{s_1}(\chi),\ldots,\Pi_{s_k}(\chi)),~~~
 \mathcal{L}_{\Pi(\chi)}:=(\mathcal{L}_{\Pi_{s_1}(\chi)
},\ldots,\mathcal{L}_{\Pi_{s_k}(\chi)}),
$$ $b:\R^{dk}\times
\mathcal{P}_2(\R^{d})^k \to\R^d$, $\si:\R^{dk}\times
\mathcal{P}_2(\R^{d})^k \to\R^{d} \otimes \R^{m}$ are given measurable functions, with $\mathcal{P}_2(\R^{d})^k := \mathcal{P}_2(\R^{d}) \times  \ldots \times \mathcal{P}_2(\R^{d})$ ($k$-times), and $W=(W(t))_{t \in [0,T]}$ is an
$m$-dimensional Brownian motion on some complete filtered atomless probability space
$(\OO,\mathcal{F},(\mathcal{F}_t)_{t\in [0,T]},\P)$, where $(\mathcal{F}_t)_{t\in [0,T]}$ is the natural filtration of $W$ augmented with an independent $\sigma$-algebra $\mathcal{F}_0$. In what follows, we will set $s_1=0$ and $s_k = -\tau$. In the sequel, we refer to these equations as point-delay McKean--Vlasov SDEs. In \cite{HUMY} the same type of delays was considered for classical SDEs.

In the following, we impose several conditions on the drift and diffusion coefficient which guarantee well-posedness of the considered point-delay McKean--Vlasov SDE.

To do so, we introduce the notation $\boldsymbol{y}_{k} := (y_1, \ldots, y_{k}), \tilde{\boldsymbol{y}}_{k} := (\tilde{y}_1, \ldots, \tilde{y}_{k}), \bar{\boldsymbol{y}}_{k} := (y'_1, y_2, \ldots, y_{k})$, for $y'_1, y_i, \tilde{y}_i \in \mathbb{R}^d$, and $\boldsymbol{\mu}_{k} := (\mu_1, \ldots, \mu_{k})$, $\tilde{\boldsymbol{\mu}}_{k} := (\tilde{\mu}_1, \ldots, \tilde{\mu}_{k})$, for $\tilde{\mu}_i, \mu_i \in \mathcal{P}_{2}(\mathbb{R}^d)$. Further, we set $\boldsymbol{0}_{k}:=(\boldsymbol{0}, \ldots, \boldsymbol{0})$, ($k$ components), where $\boldsymbol{0} \in \mathbb{R}^d$.

For the drift term $b$, we assume that:
\begin{enumerate}
\item[({\bf AD$_b^1$})]
There exist constants {\color{black}$L_b^{1}, q_1 \geq 0$} such that
\begin{align*}
& \<y_1-y'_1,b(\boldsymbol{y}_{k},\boldsymbol{\mu}_{k})-b(\bar{\boldsymbol{y}}_{k},\boldsymbol{\mu}_{k})\>\le L_b^1|y_1-y'_1|^2,  \\
& |b(\boldsymbol{y}_{k},\boldsymbol{\mu}_{k})- b(\tilde{\boldsymbol{y}}_{k},\boldsymbol{\mu}_{k})|\le L_b^1  \sum_{i=1}^{k} |y_i-\tilde{y}_i |(1+|y_i|^{q_1}+|\tilde{y}_i|^{q_1}),  \\
& |b({\bf0}_{k},\boldsymbol{\mu}_{k})- b({\bf0}_{k},\tilde{\boldsymbol{\mu}}_{k})|\le L_b^1 \sum_{i=1}^{k} \mathcal{W}_2(\mu_i,\tt \mu_i),
\end{align*}
{\color{black}for all $\bar{\boldsymbol{y}}_{k}, \boldsymbol{y}_{k}, \tilde{\boldsymbol{y}}_{k} \in \mathbb{R}^{dk}$, and
$ \boldsymbol{\mu}_{k}, \tilde{\boldsymbol{\mu}}_{k} \in \mathcal{P}_2(\R^{d})^{k}$.}
\end{enumerate}
Concerning the diffusion coefficient $\si,$ we assume:
\begin{enumerate}
\item[({\bf AD$_\sigma^1$})] 
There exist constants {\color{black}$L_\sigma^1,q_2 \geq 0$} such that
\begin{align*}
&\|\si(\boldsymbol{y}_{k},\boldsymbol{\mu}_{k})-\si(\tilde{\boldsymbol{y}}_{k},\boldsymbol{\mu}_{k})\| \le L_\sigma^1  \Big \{|y_1-\tilde{y}_1|+ \sum_{i=2}^{k} |y_i-\tilde{y}_i|(1+|y_i|^{q_2}+|\tilde{y}_i|^{q_2}) \Big \}, \\
& \|\si({\bf0}_{k},\boldsymbol{\mu}_{k}) - \si({\bf0}_{k},\tilde{\boldsymbol{\mu}}_{k})\|\le L_\si^1 \sum_{i=1}^{k} \mathcal{W}_2(\mu_i,\tt\mu_i),
\end{align*}
{\color{black}for all $\boldsymbol{y}_{k}, \tilde{\boldsymbol{y}}_{k} \in \mathbb{R}^{dk}$, and
$\boldsymbol{\mu}_{k}, \tilde{\boldsymbol{\mu}}_{k} \in \mathcal{P}_{2}(\R^{d})^{k}$.}
\end{enumerate}
We note that under these conditions, the SDE (\ref{E00}) has a unique strong solution. To present an inductive argument for this assertion, we assume, for ease of notation that there is only one delay, i.e., $k=2$ and $s_1=0, s_2=-\tau$. On the interval $[0, \tau]$, SDE (\ref{E00}) can be written as a non-delay SDE with random coefficients
\begin{equation*}
\d X(t)=b(X(t), \xi(t-\tau),\mathcal{L}_{X(t)}, \mathcal{L}_{\xi(t-\tau)}) \, \d
t+\si(X(t),\xi(t-\tau), \mathcal{L}_{X(t)},\mathcal{L}_{\xi(t-\tau)}) \, \d W(t), 
\end{equation*} 
with initial value $X(0)= \xi(0)$. This SDE has a unique strong solution under ({\bf AD$_b^1$}) and ({\bf  AD$_\sigma^1$}), see, e.g., \cite{RST}. Requiring the initial data to be in $L^{0}_{p(1+q)}(\C)$, where $q = q_1 \lor q_2$, the solution also has finite moments up to order $p$. A similar argument can be employed on the interval $[\tau,2\tau]$, and so forth up to the final time $T>0$. 
\begin{prp}\label{Prop:MomentDelayMcKean}
Let $p \geq 2$ be given. Let Assumptions ({\bf AD$_b^1$}) and ({\bf  AD$_\sigma^1$}) hold and suppose $X_0=\xi \in L^{0}_{k_1}(\C)$, where $k_1 \geq p(1+q)$, with $q = q_1 \lor q_2$, will be defined in the proof. Then, there exists a constant $C > 0$ such that
\begin{equation*}
{\color{black}\mathbb{E} \left[\sup_{- \tau \leq t \leq T} |X_t|^p \right]  \leq C(1+ \mathbb{E}[\| \xi \|^{k_1}_{\infty}]).}
\end{equation*}
\end{prp}
\begin{proof}
We prove the result for $k=2$ and $s_1=0, s_2= - \tau$ only. It\^{o}'s formula implies, for $t \geq 0$ and $p \geq 2$, 
\begin{align*}
|X(t)|^p \le & |\xi(0)|^p + p \int_{0}^{t} |X(s)|^{p-2} \left \langle X(s), b(X(s),X({s-\tau}),\mathcal{L}_{X(s)}, \mathcal{L}_{X(s-\tau)}) \right \rangle \, \mathrm{d}s \\
&+ p \int_{0}^{t} |X(s)|^{p-2} \left \langle X(s), \sigma(X(s),X({s-\tau}),\mathcal{L}_{X(s)}, \mathcal{L}_{X(s-\tau)}) \, \mathrm{d}W(s) \right \rangle  \\
&+ {\color{black}\frac{p(p-1)}{2}} \int_{0}^{t} |X(s)|^{p-2} \| \sigma(X(s),X({s-\tau}),\mathcal{L}_{X(s)}, \mathcal{L}_{X({s-\tau})}) \|^2 \, \mathrm{d}s.
\end{align*}  
Using assumptions ({\bf AD$_b^1$}), ({\bf AD$_{\sigma}^1$}) and the moment boundedness of the initial process, it is standard to show that there exists a constant $C>0$, such that
\begin{equation*}
\mathbb{E} [\|X\|_{\infty, t}^p]  \leq  C + C \int_{0}^{t} \mathbb{E} [\|X\|^p_{\infty, s}] \, \mathrm{d}s +  C \int_{0}^{0 \vee (t-\tau)} \mathbb{E}[ \|X\|_{\infty, s}^p (1+  \|X\|_{\infty, s}^{pq}) ] \, \mathrm{d}s.
\end{equation*}
We define, for the given $p \geq 2$,
\begin{equation*}
k_i:= p([T/\tau]+2-i)(1+ q)^{[T/\tau] +1- i}, \quad \text{ for } i \in \lbrace 1, \ldots, [T/\tau]+1 \rbrace.
\end{equation*}
Thus, 
\begin{equation*}
(1 + q)k_{i+1} < k_i, \quad k_{[T/\tau]+1} = p, \quad i \in \lbrace 1, \ldots, [T/\tau] \rbrace.
\end{equation*}
{\color{black}Hence, we have on $[0,\tau]$} 
\begin{equation*}
\mathbb{E}[ \|X \|_{\infty, \tau}^{k_1}]< \infty.
\end{equation*}
Consequently, by H\"{o}lder's inequality,
\begin{align*}
\mathbb{E} [ \|X \|_{\infty, 2\tau}^{k_2}]  &\leq C + C \int_{0}^{2 \tau} \mathbb{E} [\|X\|^{k_2}_{\infty, s}] \, \mathrm{d}s + C \int_{0}^{\tau} \mathbb{E}  [\|X \|_{\infty, s}^{(1 + q)k_2}]  \, \mathrm{d}s \\
&\leq  C +  C \int_{0}^{2 \tau} \mathbb{E} [\|X\|^{k_2}_{\infty, s}] \, \mathrm{d}s + C \int_{0}^{\tau} \left( \mathbb{E} [\|X \|_{\infty, s}^{k_1}]   \right)^{\frac{(1 + q)k_2}{k_1}}  \, \mathrm{d}s  < \infty.
\end{align*}
This procedure can be repeated up to time $T>0$ and the claim follows.
\end{proof}

We now introduce a particle system associated with the point-delay McKean--Vlasov SDE (\ref{E00})
\begin{align}\label{eq:DelayParticleSystem_cont}
{\rm d} X^{i,N}(t) &=  b(\Pi(X_{t}^{i,N}),\Pi(\mu_{t}^{X,N})) {\, \rm d} t + 
\sigma(\Pi(X^{i}_{t}),\Pi(\mu_{t}^{X,N})) {\, \rm d} W^{i}(t),
\end{align}
where
\begin{align*}
{\color{black}\Pi(\mu_{t}^{X,N}) := \left(\mu_{t +s_1}^{X,N}, \ldots, \mu_{t+s_k}^{X,N} \right), \quad \mu_{t +s_l}^{X,N}(\mathrm{d}x):=  \frac{1}{N} \sum_{j=1}^{N} \delta_{X^{j,N}(t+s_l)}(\mathrm{d}x), \text{ for } l \in \lbrace 1, \ldots, k \rbrace,}
\end{align*}
and $(W^i,X^{i,N}_0)_{i \in \mathbb{S}_N}$ are independent copies of $(W,\xi)$.
The following proposition provides a strong propagation of chaos result, which is proven in \cite[Proposition 5.2]{CGR}. Note that Assumptions $\bar{\bar{A}}_1$-$\bar{\bar{A}}_5$ in \cite{CGR} are satisfied if Assumptions ({\bf AD$_b^1$}) and ({\bf  AD$_\sigma^1$}) hold
(taking without loss of generality $q_1=q_2$).

\begin{prp}[Proposition 5.2 in \cite{CGR}] \label{Prop:PoCDelayMcKean}
Let Assumptions ({\bf AD$_b^1$}) and ({\bf  AD$_\sigma^1$}) hold and suppose $X_0 = \xi \in L^{0}_{k_1}(\mathscr{C})$, {\color{black}with $k_1$ as in Proposition \ref{Prop:MomentDelayMcKean} such that  $k_1 >4 $}.
Then, the interacting particle system \eqref{eq:DelayParticleSystem_cont} is strongly well-posed
and its solution $X^{i,N}$ converges to the solution $X^i$ of \eqref{E00}
with $W$ replaced by $W^i$
 (a ``non-interacting particle system''),
\[ 
\max_{i \in \mathbb{S}_N} \sup_{t\in[0,T]}\mathbb{E}[|X^i_t-X^{i,N}_t|^2]\leq C 
\begin{cases}
N^{-1/2}, &  \mbox{ if }  d<4, \\
N^{-1/2} \ln(N), &  \mbox{ if } d=4, \\
N^{-2/d}, &  \mbox{ if } d>4
\end{cases}
\]
for any $N\in\mathbb{N}$, where the constant $C>0$ does not depend on $N$. 
\end{prp}

\begin{remark}
It is possible to rewrite the particle system as a sequence of $(N\times d \times k)$-dimensional SDEs without explicit delay dependence (see \cite{KUCHLER}). Due to the dimensionality and the structure of the coefficients of the enlarged system, classical results are not directly applicable, {\color{black}as it is crucial in our analysis that implied constants are independent of the number of particles $N$} and the extra difficulties arising from the delayed components, which we address in the following, cannot be avoided.
\end{remark}

\subsection{Milstein schemes for delay McKean--Vlasov equations}\label{subsec:milstein}
We propose the following tamed Milstein schemes for the numerical approximation of the particle system \eqref{eq:DelayParticleSystem_cont},
which are motivated by the scheme presented in \cite{HUMY}. 
Considering a uniform time mesh on $[-\tau,T]$ with mesh-size $\delta = T/M = \tau/M'$, for integers $M, M'>1$, we define the following approximation of \eqref{eq:DelayParticleSystem_cont} ({\color{black}note that we assume the delayed time-points $t_n +s_l$ to be contained on the grid}): 
\begin{align}\label{eq:DelayParticleSystem}
Y^{i,N}(t_{n+1}) &=  Y^{i,N}(t_n) + b_{\delta}(\Pi(Y_{t_n}^{i,N}),\Pi(\mu_{t_n}^{Y,N}))\delta + \sigma(\Pi(Y^{i}_{t_n}),\Pi(\mu_{t_n}^{Y,N}))\Delta W^{i}_n \nonumber \\
& \quad + \sum_{l=1}^{k} \partial_{x_l} \sigma(\Pi(Y^{i,N}_{t_n}),\Pi(\mu_{t_n}^{Y,N})) \tilde{\sigma}({\color{black}t_n + s_l},\Pi(Y^{i,N}_{t_n + s_l}),\Pi(\mu_{t_n +s_l}^{Y,N}))\int_{t_n}^{t_{n+1}} \int_{t_n +s_l}^{s+ s_l}  \, \mathrm{d} W^{i}(u)  \, \mathrm{d} W^{i}(s) \nonumber \\
& \quad + \sum_{l=1}^{k} \frac{1}{N}\sum_{j = 1}^{N}  \partial_{\mu_l} \sigma(\Pi(Y_{t_n}^{i,N}), \Pi(\mu_{t_n}^{Y,N}))(Y^{j,N}(t_n +s_l)) \nonumber \\
& \hspace{2cm} \times \tilde{\sigma}({\color{black}t_n + s_l},\Pi(Y_{t_n+s_l}^{j,N}), \Pi(\mu_{t_n+s_l}^{Y,N})) \int_{t_n}^{t_{n+1}}  \int_{t_n + s_l}^{s +s_l} \, \mathrm{d}W^{j}(u)  \, \mathrm{d}W^{i}(s),
\end{align} 
for $t_n :=n\delta \geq 0$, $n \in \lbrace 0, \ldots, M-1 \rbrace$. Furthermore,
\begin{equation*}
\tilde{\sigma}({\color{black} t},\Pi(Y^{i,N}_{t}),\Pi(\mu_{t}^{Y,N})) := 
\begin{cases}
\sigma(\Pi(Y^{i,N}_{t}),\Pi(\mu_{t}^{Y,N})), \quad & t \geq 0, \\
0, \quad  & -\tau \leq t < 0,
\end{cases}
\end{equation*}
and we used the notation
\begin{align*}
 \Pi(\mu_{t_n}^{Y,N}) := \left(\mu_{t_n +s_1}^{Y,N}, \ldots, \mu_{t_n+s_k}^{Y,N} \right), \quad \mu_{t_n +s_l}^{Y,N}(\mathrm{d}x):=  \frac{1}{N} \sum_{j=1}^{N} \delta_{Y^{j,N}(t_n+s_l)}(\mathrm{d}x), \text{ for } l \in \lbrace 1, \ldots, k \rbrace.
\end{align*}
We denote by $\partial_{x_l}$ the gradient with respect to the $l$-th state component and by $\partial_{\mu_l}$ the $L$-derivative with respect to $l$-th measure component. Observe that $\partial_{x_l}$ and $\partial_{\mu_l}$ when applied to $\sigma$ have to be understood in an operator sense. Note that for $t \leq 0$, we have $Y^{i,N}(t) = \xi^i(t)$; an independent copy of $\xi$. As in the non-delay case \cite{stock}, we define $b_\dd$ in two different ways yielding schemes which will be denoted by Scheme $1$ and Scheme $2$, respectively: 
For Scheme 1, we use  
\begin{equation*}
b_\dd(x,\mu):=\ff{b(x,\mu)}{1+\dd|b(x,\mu)|}, \quad x \in \R^{dk}, \
\mu \in \mathcal {P}_2(\R^{d})^k,
\end{equation*}
and for Scheme 2, we define  
\begin{equation*}
b_\dd(x,\mu):=\ff{b(x,\mu)}{1+\dd|b(x,\mu)|^{\bar{q}}}, \quad x \in \R^{dk}, \
\mu\in\mathcal {P}_2(\R^{d})^k,
\end{equation*}
where $\bar{q} = \tfrac{2q_1}{q_1+1}$. For any $t \in [0,T]$, let $t_\dd :=  \lfloor t / \dd \rfloor \delta$. The continuous-time version of (\ref{eq:DelayParticleSystem}) is
\begin{align*}
\mathrm{d}Y^{i,N}(t) &=  b_{\delta}(\Pi(Y^{i,N}_{t_\dd}),\Pi(\mu_{t_\dd}^{Y,N}))\, \mathrm{d}t + \Bigg(\sigma(\Pi(Y^{i}_{t_\dd}),\Pi(\mu_{t_\dd}^{Y,N}))  \nonumber \\
& \quad + \sum_{l=1}^{k} \partial_{x_l} \sigma(\Pi(Y^{i,N}_{t_\dd}),\Pi(\mu_{t_\dd}^{Y,N})) \tilde{\sigma}({\color{black}t_\dd + s_l},\Pi(Y^{i,N}_{t_\dd + s_l}),\Pi(\mu_{t_\dd +s_l}^{Y,N})) \int_{t_\dd +s_l}^{t+ s_l}  \, \mathrm{d} W^{i}(u)  \nonumber \\
& \quad + \sum_{l=1}^{k} \frac{1}{N}\sum_{j = 1}^{N}  \partial_{\mu_l} \sigma(\Pi(Y_{t_\dd}^{i,N}), \Pi(\mu_{t_\dd}^{Y,N}))(Y^{j,N}(t_\dd +s_l))  \nonumber \\
& \hspace{2cm} \times \tilde{\sigma}({\color{black}t_\dd + s_l},\Pi(Y_{t_\dd+s_l}^{j,N}), \Pi(\mu_{t_\dd+s_l}^{Y,N}))  \int_{t_\dd + s_l}^{t +s_l} \, \mathrm{d}W^{j}(u) \Bigg)  \, \mathrm{d}W^{i}(t).
\end{align*}
A Milstein scheme for globally Lipschitz continuous drift and diffusion coefficients can be readily formulated by replacing $b_\dd$ with $b$ in (\ref{eq:DelayParticleSystem}).

In the sequel, we restrict the discussion to $d=m=1$, and introduce
\begin{equation*}
I^{i}(t_n + s_l,t_{n+1} +s_l; s_l) := \int_{t_n}^{t_{n+1}} \int_{t_n +s_l}^{s+ s_l}  \, \mathrm{d} W^{i}(u)  \, \mathrm{d} W^{i}(s). 
\end{equation*}
For $s_l = 0$, this double It\^{o} integral simplifies to 
\begin{equation*}
I^{i}(t_n,t_{n+1};0) = \frac{1}{2}( (\Delta W^{i}_n)^2 -\delta).
\end{equation*}
To simulate all the double stochastic integrals for $s_l \neq 0$, which appear for delay equations already for $m=d=1$, one can employ an approximation of $I^{i}(t_n + s_l,t_{n+1} +s_l; s_l)$; see \cite{KP} and \cite{WSON,JGTL,SJAMAW} for further developments on this topic. 
{\color{black} Based on the definitions $\bar{W}^{i}(s):=W^{i}(s+t_n + s_l) - W^{i}(t_n + s_l)$ (delayed Brownian increment), and $\bar{B}^{i}(s):=\bar{W}^{i}(s-s_l) - \bar{W}^{i}(-s_l)$ (present increment), for $s \geq 0$, \cite{HUMY} introduces an approximation for $I^{i}(t_n + s_l,t_{n+1} +s_l; s_l)$ based on an infinite series representation for Brownian bridges (truncated at level $r \in \mathbb{N}$) as in \cite{KP}.
In our tests, we add a term for approximating the tail of the L\'{e}vy area as in \cite{WSON}, to improve the mean-square error of the approximation to the order $\mathcal{O}(\delta^2/r^2)$. This justifies the choice $r = \mathcal{O}(\delta^{-1/2})$, compared to 
$1/r \leq \delta \land \min \lbrace |s_l|: \ l \in \lbrace 1, \ldots, k \rbrace  \rbrace$
 given in \cite[Lemma 4.2]{HUMY}, and hence reduces the complexity.} 
The tamed Milstein scheme combined with this approximation technique of the double stochastic integrals has computational complexity of order $\vv^{-3/2}$
per sample for a root-mean-square error of order $\vv>0$, compared to the standard Euler--Maruyama scheme, which has complexity of order $\vv^{-2}$. \\ \\
\noindent
The following additional assumptions will be needed for the subsequent presentation and are motivated by the assumptions formulated in \cite{stock}. {\color{black}We impose the following to hold uniformly in $x_1, x_1', z_1, z_1' \in\R^d$, $\mu_1, \tilde{\mu}_1 \in \mathcal{P}_2(\R^d)$, $\boldsymbol{y}_{k}, \tilde{\boldsymbol{y}}_{k}, \hat{\boldsymbol{y}}_{k} \in \R^{dk}$, and $\boldsymbol{\nu}_{k}, \tilde{\boldsymbol{\nu}}_{k}, \hat{\boldsymbol{\nu}}_{k} \in \mathcal{P}_{2}(\R^{d})^{k}$}:
\begin{enumerate}
\item[({\bf AD$_b^2$})] The drift $b$ does not depend on delay variables and $b \in C^{1,1}(\R^{d}\times \mathcal{P}_2(\R^{d}))$, i.e., $b$ is continuously differentiable in both components. 
\item[({\bf AD$_b^3$})] There exist constants  $ L^3_b, q_3  \geq 0$ such that
\begin{equation*}
\begin{split}
\|\partial_x b(x_1,\mu_1)- \partial_x b(x_1',\tilde{\mu}_1)\| &\le
L_b^3 \Big \{ |x_1-x_1'|(1 + |x_1|^{q_3} + |x_1'|^{q_3}) + \mathcal{W}_2(\mu_1,\tilde{\mu}_1) \Big \},\\
\|\partial_{\mu}b(x_1,\mu_1)(z_1)-\partial_{\mu}b(x_1',\tilde{\mu}_1)(z_1') \| &\le
L_b^3 \Big \{|x_1-x_1'| + |z_1-z_1'| + \mathcal{W}_2(\mu_1, \tilde{\mu}_1) \Big \}.
\end{split}
\end{equation*}
\end{enumerate}
\noindent
The reason why $b$ cannot depend on delay variables will be made clearer in the proof of Theorem \ref{THDelay:THDelay4}.
Concerning the diffusion term $\si$ with columns denoted by $\si_u$ for $u  \in \lbrace 1, \ldots, m \rbrace$, we further require:
\begin{enumerate}
\item[({\bf AD$_\sigma^2$})] {\color{black}For all $u \in  \lbrace 1, \ldots, m \rbrace$, we have $\sigma_u \in C^{1,1}(\R^{dk}\times \mathcal{P}_2(\R^{d})^k)$ and, in particular, there exist constants $L_{\sigma}^2, q_2 \geq 0$ such that for all $u$
\begin{equation*}
\begin{split}
& \|  \partial_{\mu_1} \sigma_u(\boldsymbol{y}_{k},\boldsymbol{\nu}_{k})(z_1) \sigma(\boldsymbol{y}_{k},\boldsymbol{\nu}_{k})  \| \ \vee \  \| \partial_{x_1} \sigma_u(\boldsymbol{y}_{k},\boldsymbol{\nu}_{k}) \sigma(\boldsymbol{y}_{k}, \boldsymbol{\nu}_{k})  \|  \\
& \qquad \leq L_{\sigma}^2 \Big \lbrace 1+|z_1| +|y_1|+\nu_1(|\cdot|^{2})^{\frac{1}{2}}+ \sum_{i=2}^{k} ( |y_i|^{q_2+1} + \nu_i(|\cdot|^{2})^{\frac{1}{2}}) \Big \rbrace.
\end{split}
\end{equation*}}
For the derivatives with respect to the delay variables, we require that one of the following holds (the reason and structure for these assumptions will become clearer in Section \ref{Section:Sec3}):
\begin{enumerate}[label=(\alph*)]
\item The derivatives (in the state and measure component) with respect to the delay variables are uniformly bounded. In particular, we impose $q_2=0$ in ({\bf AD$_\sigma^1$}). 
\item {\color{black}For all $u \in  \lbrace 1, \ldots, m \rbrace$ and for any $l>1$, and $z_l \in \mathbb{R}^{d}$, the derivatives $ \partial_{\mu_l} \sigma_u(\boldsymbol{y}_{k}, \boldsymbol{\nu}_{k})(z_l)$ and $\partial_{x_l} \sigma_u(\boldsymbol{y}_{k},\boldsymbol{\nu}_{k})$ do not depend on $y_1, \ldots, y_{l-1}$ and $\nu_1, \ldots, \nu_{l-1}$, and for a constant $q_3 \geq 0$
\begin{align*}
& \|  \partial_{\mu_l} \sigma_u(\boldsymbol{y}_{k},\boldsymbol{\nu}_{k})(z_l) \sigma(\hat{\boldsymbol{y}}_{k}, \hat{\boldsymbol{\nu}}_{k})  \|  \\
& \quad \leq L_{\sigma}^2 \Big \lbrace 1+ |z_l| + \sum_{i=l}^{k} ( |y_i|^{q_3+1} + \nu_i(|\cdot|^{2})^{\frac{1}{2}})   + \sum_{i=1}^{k} ( |\hat{y}_i|^{q_3+1} + \hat{\nu}_i(|\cdot|^{2})^{\frac{1}{2}}) \Big \rbrace, \\
& \| \partial_{x_l} \sigma_u(\boldsymbol{y}_{k},\boldsymbol{\nu}_{k}) \sigma(\hat{\boldsymbol{y}}_{k},\hat{\boldsymbol{\nu}}_{k}) \| \\
& \quad \le L_{\sigma}^2 \Big \lbrace 1+ \sum_{i=l}^{k} ( |y_i|^{q_3+1} + \nu_i(|\cdot|^{2})^{\frac{1}{2}})  + \sum_{i=1}^{k} ( |\hat{y}_i|^{q_3+1} + \hat{\nu}_i(|\cdot|^{2})^{\frac{1}{2}}) \Big \rbrace.
\end{align*}
}
\end{enumerate}
\item[({\bf AD$_\sigma^3$})] There exists a constant $L_{\sigma}^3 \geq 0$ such that for all $u \in  \lbrace 1, \ldots, m \rbrace$
\begin{equation*}
\begin{split}
&\|\partial_{x_1} \si_u(\boldsymbol{y}_{k},\boldsymbol{\nu}_{k})-\partial_{x_1} \si_u(\tilde{\boldsymbol{y}}_{k}, \tilde{\boldsymbol{\nu}}_{k}) \| \le L_{\sigma}^3 \Big \{ \sum_{i=1}^{k}(|y_i-\tilde{y}_i| + \mathcal{W}_{2}(\nu_i, \tilde{\nu}_i))\Big \}, \\
&\|  \partial_{\mu_1} \si_u(\boldsymbol{y}_{k},\boldsymbol{\nu}_{k})(z_1)-  \partial_{\mu_1} \si_u(\tilde{\boldsymbol{y}}_{k}, \tilde{\boldsymbol{\nu}}_{k})(z_1')\| \le L_{\sigma}^3 \Big \{ |z_1 - z_1'| + \sum_{i=1}^{k}(|y_i - \tilde{y}_i| + \mathcal{W}_{2}(\nu_i, \tilde{\nu}_i))  \Big \}, \\
\end{split}
\end{equation*}
and analogously for the derivatives in the delay components.
\end{enumerate}
\noindent
In addition to above assumptions, we demand:
\begin{itemize}
\item[({\bf H$_1$})] Regularity on the initial data $\xi$: $\xi$ is a deterministic function and there exists a constant $C>0$, {\color{black}such that for any $s,t \in [-\tau,0]$
\begin{equation}\label{eq.MallBound}
| \xi(t) - \xi(s) | \leq C|t-s|.
\end{equation}}
\end{itemize}
The assumption that $\xi$ is deterministic is made to simplify the formulation of the results and could be relaxed. 
We give a stability and strong convergence analysis for Scheme 2 only, as less assumptions are required for this scheme (compared to Scheme 1); see \cite{stock} for a discussion on this. We prove the following statement: 

\begin{thm}\label{THDelay:THDelay4}
Let Assumptions ({\bf AD$_b^1$})--({\bf AD$_b^3$}), ({\bf AD$_\sigma^1$})--({\bf AD$_\sigma^3$}), and ({\bf H$_1$}) hold. Then, for any $p \geq 2$ there exist constants $C_1, C_2 >0$ (independent of $M$ and $N$) such that
\begin{equation*} 
   \max_{i \in \mathbb{S}_N} \max_{n \in \lbrace 0, \ldots, M \rbrace} \mathbb{E}[| Y^{i,N}(t_n) |^p]  \leq C_1, 
\end{equation*} 
and 
\begin{equation}\label{eq:ResultDelay2}
\max_{i \in \mathbb{S}_N} \E[\|X^{i,N}-Y^{i,N}\|^p_{\8, T}]\le C_2 \dd^p,
\end{equation}
where $X^{i,N}$ is given by \eqref{IPSD} and $Y^{i,N}$ is the continuous time version of (\ref{eq:DelayParticleSystem}) with $b_\dd(x,\mu):=\ff{b(x,\mu)}{1+\dd|b(x,\mu)|^{\bar{q}}}$. 
\end{thm}
\begin{proof}
The proof is deferred to Section \ref{Section:Sec3}.
\end{proof}

\begin{remark}
Since the implied constant in Theorem \ref{THDelay:THDelay4} is independent of $N$, we can easily complete our statement with 
Proposition \ref{Prop:PoCDelayMcKean} to obtain an approximation of \eqref{E00} in an $L_p$-sense. {\color{black}However, note that the rate provided in Proposition \ref{Prop:PoCDelayMcKean} is, in general, not optimal and does not yield the $1/N$-rate given in \eqref{eq:VarAssump}. To achieve such a convergence order further structural or regularity assumption on the coefficients of \eqref{E00} are required, see e.g., \cite{MEL,LSAT}.}  
\end{remark}

\section{Antithetic MLMC approach for delay equations}\label{Section:SectionAnti}

As explained in the introduction and detailed in the previous section, 
Milstein schemes for delay equations require computationally costly simulation of iterated stochastic integrals already for $d=m=1$. This motivates the study of an antithetic approach for such delay McKean--Vlasov equations in order to improve the efficiency of the simulation. 


\subsection{Problem formulation}
We consider the following one-dimensional, one-point delay McKean--Vlasov SDE
\begin{equation}\label{delayMcKeanequation}
\mathrm{d}X(t) = b(X(t),\mathcal{L}_{X(t)}) \, \mathrm{d}t + \sigma(X(t), X(t- \tau)) \, \mathrm{d}W(t), \quad X_0=\xi \in L^{0}_p(\C),
\end{equation}
for a given $p \geq 2$, with measurable functions $b:\R \times
\mathcal{P}_2(\R)\to \R$, $\si:\R^{2} \to \R$, and where $(W(t))_{t \in [0,T]}$ is a
one-dimensional Brownian motion on some filtered probability space $(\Omega,\mathcal{F},(\mathcal{F}_t)_{t \in [0,T]},\mathbb{P})$. {\color{black}As earlier, we will restrict the discussion to deterministic initial data with time regularity specified in (\ref{eq.MallBound}).} 

{\color{black}
Above and in the analysis, we assume that the diffusion coefficient has no law dependence.
We present the extension of the scheme to the law dependent case below, but for the main part avoid adding to the already lengthy notation.} The scheme is directly applicable to equations with delay in the drift $b$, but the analysis of this case does not follow directly and is left for future work.


The associated particle system has the form 
\begin{equation}\label{delayPS}
\mathrm{d}X^{i,N}(t) = b(X^{i,N}(t),\mu_t^{X,N}) \, \mathrm{d}t + \sigma(X^{i,N}(t), X^{i,N}(t- \tau)) \, \mathrm{d}W^{i}(t), \quad X_0^{i,N}=\xi^{i} \in \C,
\end{equation}
with $(W^{i},\xi^{i})$ an independent copy of $(W,\xi)$ for $i \in \mathbb{S}_N$, and 
\begin{equation*}
\mu_t^{X,N}(\mathrm{d}x):= \frac{1}{N} \sum_{j=1}^{N} \delta_{X^{j,N}(t)}(\mathrm{d}x), \quad t \in [0,T].
\end{equation*} 
In the sequel, we will present model assumptions that guarantee well-posedness of the one point-delay McKean--Vlasov SDE (and the associated particle system) and will also be employed for the analysis of the subsequent numerical scheme. {\color{black}We impose the following to hold uniformly in $x,y,x',y' \in\R $ and $\mu,\nu \in \mathcal{P}_2(\R)$.}
\begin{enumerate}
\item[({\bf AAD}$_b^1$)] There exist constants $L_b^1,q_1 \geq0$ such that
\begin{align*}
&\<x-y,b(x,\mu)-b(y,\mu)\>\le L_b^1 |x-y|^2,  \\
&|b(x,\mu)-b(y,\mu)|\le L_b^1(1+|x|^{q_1}+|y|^{q_1})|x-y|, \\
&|b(x,\mu)-b(x,\nu)|\le L_b^1\mathcal{W}_2(\mu,\nu).
\end{align*}
\item [({\bf AAD}$_b^2$)] The drift satisfies assumption ({\bf AD}$_b^3$).
\item [({\bf AAD}$_b^3$)] The function $\R \ni x \mapsto b(x, \mu)$ belongs to $C^{2}(\R;\R)$ and the second order derivative is (possibly) of polynomial growth, i.e., there exist constants $L^{3}_b \geq 0$ and $q_2 \geq 0$ such that 
\begin{align*}
|\partial^2_x b(x,\mu)| \leq L^{3}_b(1 + |x|^{1 + q_2}).
\end{align*} 
\item[({\bf AAD}$_\si^1$)] There exists an $L_\si^1 \geq 0$ such that
\begin{equation*}
|\si(x,y)-\si(x',y') | \le L_\si^1(|x-x'| + |y-y'|).
\end{equation*}
\item[({\bf AAD}$_\si^2$)] We have $\R \ni x \mapsto \sigma(x, y), \ \R \ni y \mapsto \sigma(x,y) \in C^{2}(\R;\R)$ and their second order derivatives are uniformly bounded.
\end{enumerate}
\subsection{Antithetic sampling scheme}
Throughout the remaining sections, we denote by $\partial_{x_1} \sigma$ and $\partial_{x_2} \sigma$ the derivatives of $\sigma$ with respect to the first and second component, respectively.

The tamed Milstein scheme for this system of particles has the form ({\color{black}for ease of notation, we refrain from using to notation $\tilde{\sigma}$ below})
\begin{align*}
{Y}^{i,N}({t_{n+1}}) & = {Y}^{i,N}({t_{n}}) +   b_\dd({Y}^{i,N}(t_n), \mu_{t_n}^{{Y},N}) \delta + \sigma({Y}^{i,N}(t_n), {Y}^{i,N}(t_n- \tau)) \Delta W^{i}_n  \\
& \quad + \sigma({Y}^{i,N}(t_n), {Y}^{i,N}(t_n- \tau)) \partial_{x_1} \sigma({Y}^{i,N}(t_n), {Y}^{i,N}(t_n- \tau)) \frac{(\Delta W^{i}_n)^2- \delta}{2} \\
& \quad + \sigma({Y}^{i,N}(t_n- \tau), {Y}^{i,N}(t_n- 2\tau)) \partial_{x_2} \sigma({Y}^{i,N}(t_n), {Y}^{i,N}(t_n- \tau)) \int_{t_n}^{t_{n+1}} \int_{t_n}^{s} \, \mathrm{d}W^{i}_{u- \tau} \, \mathrm{d}W^{i}_s
\end{align*}
for $t_m:=m\delta$, $m \in \lbrace 0, \ldots, M \rbrace$, with $\delta = T/M = \tau/M'$ for $M, M' >1$,
and where
\begin{equation*}
{Y}^{i,N}(\theta) = \xi^{i}(\theta), \theta \in [-\tau,0], \quad  \mu_{t_n}^{{Y},N}(\mathrm{d}x) := \frac{1}{N} \sum_{j=1}^{N} \delta_{{Y}^{j,N}({t_{n}})}(\mathrm{d}x).
\end{equation*}
For simplicity, we assume that $b_\dd$ is defined by Scheme 2 introduced in Section \ref{subsec:milstein}, as less assumptions are needed to guarantee moment boundedness in that case.

Now, we propose a modified scheme which does not require the simulation of iterated stochastic integrals, but still achieves a higher order convergence rate for the variance of multilevel correction terms: We define for $n \in \lbrace 0, \ldots, M-1 \rbrace$ what will be the coarse grid solution
\begin{align}\label{eq:coarsepath}
Y^{i,N,c}({t_{n+1}}) & = Y^{i,N,c}({t_{n}}) + b_\dd(Y^{i,N,c}(t_n), \mu_{t_n}^{Y,N,c}) \delta + \sigma(Y^{i,N,c}(t_n), Y^{i,N,c}(t_n- \tau)) \Delta W^{i}_n  \nonumber \\
& \quad + \sigma(Y^{i,N,c}(t_n), Y^{i,N,c}(t_n- \tau)) \partial_{x_1} \sigma (Y^{i,N,c}(t_n), Y^{i,N,c}(t_n- \tau)) \frac{(\Delta W^{i}_n)^2- \delta}{2} \nonumber  \\
& \quad + \sigma(Y^{i,N,c}(t_n- \tau), Y^{i,N,c}(t_n- 2\tau)) \partial_{x_2} \sigma (Y^{i,N,c}(t_n), Y^{i,N,c}(t_n- \tau))\frac{\Delta W^{i}_n \Delta W^{i}_{n- \tau}}{2}, \nonumber \\
&=: \mathscr{S}\left(Y^{i,N,c}({t_{n}}), Y^{i,N,c}({t_{n}-\tau}), Y^{i,N,c}({t_{n}- 2\tau}), \mu_{t_n}^{Y,N,c}, \delta, \Delta W^{i}_n, \Delta W^{i}_n \Delta W^{i}_{n- \tau} \right),
\end{align}
where
\begin{equation*}
 Y^{i,N,c}(\theta) := \xi^{i}(\theta), \theta \in [-\tau,0], \quad \mu_{t_n}^{Y,N,c}(\mathrm{d}x):= \frac{1}{N} \sum_{j=1}^{N} \delta_{Y^{j,N,c}({t_{n}})}(\mathrm{d}x),
\end{equation*}
and $\Delta W^{i}_{n- \tau} := W^{i}(t_{n+1} - \tau) - W^{i}(t_{n} - \tau)$.

{\color{black}The above proposed modified scheme on the coarse grid can be readily extended to the setting where the diffusion is law dependent, i.e., as in \eqref{E00},
\begin{eqnarray*}
\d X(t)=b(\Pi(X_t),\mathcal{L}_{\Pi(X_t)}) \, \d
t+\si(\Pi(X_t),\mathcal{L}_{\Pi(X_t)}) \, \d W(t), \quad X_0=\xi\in \C.
\end{eqnarray*}
To be precise, consider a diffusion depending explicitly on $\mathcal{L}_{X(t)}$ and $\mathcal{L}_{X(t-\tau)}$ (i.e., $k=2$), then this would yield the following modified scheme: For $n \in \lbrace 0, \ldots, M-1 \rbrace$
\begin{align*}
Y^{i,N,c}({t_{n+1}}) & = Y^{i,N,c}({t_{n}}) + b_\dd(Y^{i,N,c}(t_n), \mu_{t_n}^{Y,N,c}) \delta + \sigma(\Pi(Y^{i,N,c}_{t_n}),\Pi(\mu_{t_n}^{Y,N,c})) \Delta W^{i}_n  \nonumber \\
& \quad + \sigma(\Pi(Y^{i,N,c}_{t_n}),\Pi(\mu_{t_n}^{Y,N,c})) \partial_{x_1} \sigma(\Pi(Y^{i,N,c}_{t_n}),\Pi(\mu_{t_n}^{Y,N,c})) \frac{(\Delta W^{i}_n)^2- \delta}{2} \nonumber  \\
& \quad + \sigma(\Pi(Y^{i,N,c}_{t_n- \tau}),\Pi(\mu_{t_n-\tau}^{Y,N,c})) \partial_{x_2} \sigma (\Pi(Y^{i,N,c}_{t_n}),\Pi(\mu_{t_n}^{Y,N,c}))\frac{\Delta W^{i}_n \Delta W^{i}_{n- \tau}}{2}, \nonumber\\
& \quad + \frac{1}{N} \sum_{j=1}^{N} \sigma(\Pi(Y^{j,N,c}_{t_n}),\Pi(\mu_{t_n}^{Y,N,c})) \partial_{\mu_1} \sigma (\Pi(Y^{i,N,c}_{t_n}),\Pi(\mu_{t_n}^{Y,N,c}))(Y^{j,N,c}(t_n)) \frac{\Delta W^{i}_n \Delta W^{j}_n- \mathrm{I}_{\lbrace i=j \rbrace}\delta}{2} \nonumber  \\
& \quad + \frac{1}{N} \sum_{j=1}^{N} \sigma(\Pi(Y^{j,N,c}_{t_n- \tau}),\Pi(\mu_{t_n-\tau}^{Y,N,c})) \partial_{\mu_2} \sigma (\Pi(Y^{i,N,c}_{t_n}),\Pi(\mu_{t_n}^{Y,N,c}))(Y^{j,N,c}(t_n-\tau)) \frac{\Delta W^{i}_n \Delta W^{j}_{n-\tau}}{2}, 
\end{align*}
where
\begin{equation*}
 Y^{i,N,c}(\theta) := \xi^{i}(\theta), \theta \in [-\tau,0], \quad \mu_{t_n}^{Y,N,c}(\mathrm{d}x):= \frac{1}{N} \sum_{j=1}^{N} \delta_{Y^{j,N,c}({t_{n}})}(\mathrm{d}x)
\end{equation*}
(and similarly the processes $Y^{i,N,f}$ and $Y^{i,n,a}$, precisely defined below). 
The presence of $L$-derivative terms and associated iterated integrals can be accounted for in the proofs as the terms involving spatial derivatives without providing new mathematical challenges.
}

Using above assumptions, it follows that the moments of $Y^{i,N,c}$ are bounded and that $Y^{i,N,c}$ converges strongly to the solution of SDE (\ref{delayPS}), with strong order 1/2; see Lemma \ref{lem_coarse} below.

In the sequel, $\delta$ will correspond to the mesh-size of a coarse time discretisation, where the elements of the time-grid are denoted by $t_0, t_1, \ldots, t_M$. The corresponding fine time-grid will additionally include the elements $t_{1/2}, t_{3/2}, \ldots, t_{M-1/2}$, where $t_{n + 1/2}:= \left(n + 1/2 \right) \delta$ for $n \in \lbrace 0, \ldots, M-1 \rbrace$. Hence, we will make use of the notation $\delta W_n := W(t_{n+1/2}) - W(t_{n})$ and $\delta W_{n+1/2} := W(t_{n+1}) - W(t_{n+1/2})$. Analogously, we can define $\delta W_{n- \tau}$ and $\delta W_{n+1/2 - \tau}$.

On a refined mesh with step-size $\delta/2$, we first introduce two discrete processes $Y^{i,N,f}$ and $Y^{i,N,a}$
\begin{align}\label{antithetic:f}
& Y^{i,N,f}({t_{n+1/2}}) = \mathscr{S}\left(Y^{i,N,f}({t_{n}}), Y^{i,N,f}({t_{n}-\tau}), Y^{i,N,f}({t_{n}- 2\tau}), \mu_{t_n}^{Y,N,f}, \delta/2, \delta W^{i}_n, \delta W^{i}_n  \delta W^{i}_{n- \tau} \right), \notag \\
&Y^{i,N,f}({t_{n+1}}) \\
& \;\; = \mathscr{S}\left(Y^{i,N,f}({t_{n+1/2}}), Y^{i,N,f}({t_{n+1/2}\!-\!\tau}), Y^{i,N,f}({t_{n+1/2}\!-\! 2\tau}), \mu_{t_{n+1/2}}^{Y,N,f}, \delta/2, \delta W^{i}_{n+1/2}, \delta W^{i}_{n+1/2}  \delta W^{i}_{n+1/2- \tau} \right) \notag
\end{align}
for $n \in \lbrace 0, \ldots, M-1 \rbrace$, with 
\begin{equation*}
 \mu_{t'}^{Y,N,f}(\mathrm{d}x) := \frac{1}{N} \sum_{j=1}^{N} \delta_{Y^{j,N,f}({t'})}(\mathrm{d}x), \text{ where } t'=t_n, \ t_{n+1/2},
\end{equation*}
and then the antithetic version of this process
\begin{align}\label{antithetic:a}
Y^{i,N,a}({t_{n+1/2}}) &= \mathscr{S}\left(Y^{i,N,a}({t_{n}}), Y^{i,N,a}({t_{n}\!-\!\tau}), Y^{i,N,a}({t_{n}\!-\! 2\tau}), \mu_{t_{n}}^{Y,N,a}, \delta/2, \delta W^{i}_{n+1/2}, \delta W^{i}_{n+1/2}  \delta W^{i}_{n+1/2- \tau} \right), \\
Y^{i,N,a}({t_{n+1}}) &= \mathscr{S}\left(Y^{i,N,a}({t_{n+1/2}}), Y^{i,N,a}({t_{n+1/2}\!-\!\tau}), Y^{i,N,a}({t_{n+1/2}\!-\! 2\tau}), \mu_{t_{n+1/2}}^{Y,N,a}, \delta/2, \delta W^{i}_{n}, \delta W^{i}_{n}  \delta W^{i}_{n- \tau} \right) \notag
\end{align}
with
\begin{equation*}
  \mu_{t'}^{Y,N,a}(\mathrm{d}x) := \frac{1}{N} \sum_{j=1}^{N} \delta_{Y^{j,N,a}({t'})}(\mathrm{d}x), \text{ where } t'=t_n, \ t_{n+1/2}.
\end{equation*}

We now take the average of $Y^{i,N,f}$ and $Y^{i,N,a}$ and define for all {\color{black}$n \in \lbrace 0, \ldots, M \rbrace$}
\begin{equation}
\label{Y_bar}
\overline{Y}^{i,N,f}(t_{n}) := \frac{Y^{i,N,f}({t_{n}})+Y^{i,N,a} ({t_{n}})}{2}.
\end{equation}
Similarly, we define estimators to be used in conjunction with \eqref{eqn:MLML} as follows:
\begin{eqnarray*}
\phi_N^l := \frac{1}{N} \sum_{j=1}^N \tfrac{P(Y^{j,N,f}(T)) + P(Y^{j,N,a}(T))}{2}, \qquad
\varphi_N^{l-1} := \frac{1}{N} \sum_{j=1}^N P(Y^{j,N,c}(T)).
\end{eqnarray*}


First, we give a lemma on stability and strong convergence of the coarse grid solution as a consequence of the proofs of the results stated in Section \ref{Section:Sec2}.
\begin{lem}\label{lem_coarse}
Let Assumptions ({\bf AAD}$_b^1$)--({\bf AAD}$_b^3$) and ({\bf AAD}$_\si^1$)--({\bf AAD}$_\si^2$) hold. For $n \in \lbrace 0, \ldots, M \rbrace$, let $Y^{i,N,c}(t_n)$ be defined as above and $X^{i,N}(t_n)$ be given by (\ref{delayPS}). Then, for any $p  \geq 2$ there exists a constant $C>0$ such that  
\begin{align*}
& \max_{i \in \mathbb{S}_N} \max_{n \in \lbrace 0, \ldots, M \rbrace}  \mathbb{E} \left[|Y^{i,N,c}(t_n)|^{p} \right] \leq C, \qquad \max_{i \in  \mathbb{S}_N} \max_{n \in \lbrace 0, \ldots, M \rbrace} \mathbb{E}[|Y^{i,N,c}(t_n) - X^{i,N}(t_n)|^{p}] \leq C \delta^{p/2}. 
\end{align*}
\end{lem}


The following main result reveals the first order convergence of the difference between fine and coarse grid approximation for the antithetic scheme.
\begin{thm}\label{TheoremAntithetic}
Let Assumptions ({\bf AAD}$_b^1$)--({\bf AAD}$_b^3$) and ({\bf AAD}$_\si^1$)--({\bf AAD}$_\si^2$) hold. For $n \in \lbrace 0, \ldots, M \rbrace$, let $\overline{Y}^{i,N,f}(t_{n})$ be defined as in \eqref{Y_bar} and $Y^{i,N,c}(t_{n})$ be the coarse path approximation, defined as in (\ref{eq:coarsepath}). Then, there exists a constant $C>0$ such that 
\begin{align}
\label{main_antith}
\max_{i \in \mathbb{S}_N} \max_{n \in \lbrace 0, \ldots, M \rbrace} \mathbb{E} [|\overline{Y}^{i,N,f}(t_{n}) - Y^{i,N,c}(t_{n})|^2] \leq C \delta^{2}. 
\end{align}
\end{thm}
\begin{proof}
The proof is deferred to Section \ref{panti}.
\end{proof}

\subsection{MLMC complexity analysis}\label{sec:complex}

To investigate the cost of the proposed MLMC estimator in combination with the antithetic approach, we state for the reader's convenience an adaptation of the complexity result of \cite{AT} to our setting.

\begin{prp}[cf.\ Theorem 3.1 in \cite{AT} ] \label{TH:MLMCCali}
For every $(\tilde{L},l) \in (\mathbb{N}\cup \{0\})^2$, let $N=\tilde{\beta}^{\tilde{L}}$ for
some $\tilde{\beta}>0$,
and $\phi_{N}^l$, $\varphi_{N}^{l-1}$ be approximations of the random variable $P$.
For $k \in \{0,\ldots,K_l\}$, let $\phi_{N}^{l}(\boldsymbol{\omega}^{(l,k)}_{1:N})$ and $\varphi_{N}^{l-1}(\boldsymbol{\omega}^{(l,k)}_{1:N})$ be the $k$-th realisations
of $\phi_{N}^l$ and $\varphi_{N}^{l-1}$, respectively.
Consider the MLMC estimator 
\begin{equation*}
\mathcal{A}_{\text{MLMC}}(\tilde{L},L) = \sum_{l=0}^{L} \frac{1}{K_l} \sum_{k=1}^{K_l} \left(\phi_{N}^{l}(\boldsymbol{\omega}^{(l,k)}_{1:N}) - \varphi_{N}^{l-1}(\boldsymbol{\omega}^{(l,k)}_{1:N})\right),
\end{equation*}
with $\varphi_{N}^{-1} =0$ and for $\beta, w, \gamma, s, \tilde{\gamma}, \tilde{w}, \tilde{c} >0$, where $s \leq 2w$, assume the following:
\begin{enumerate}
\item[(i)] $\left | \mathbb{E}\left[P - \varphi_{N}^{l} \right] \right | = \mathcal{O}\left(\tilde{\beta}^{-\tilde{w} \tilde{L}} + \beta^{-wl} \right)$ 
\item[(ii)] $\mathbb{E}\left[ \phi_{N}^{l} \right] = \mathbb{E}\left[  \varphi_{N}^{l} \right]$
\item[(iii)] $\mathbb{V}\left[\phi_{N}^{l} - \varphi_{N}^{l-1}  \right] = \mathcal{O}\left(\tilde{\beta}^{-\tilde{c}\tilde{L}} \beta^{-sl} \right)$
\item[(iv)] $\text{Work} \left[\phi_{N}^{l} - \varphi_{N}^{l-1}  \right] = \mathcal{O} \left(\tilde{\beta}^{\tilde{\gamma}\tilde{L}} \beta^{\gamma l} \right)$.
\end{enumerate}
Then, for any $\varepsilon < e^{-1}$, there exist $\tilde{L}, L$ and a sequence of $(K_l)_{l \in \lbrace 0, \ldots, L \rbrace }$, such that
\begin{equation*}
MSE := \mathbb{E} \left[ (\mathcal{A}_{\text{MLMC}}(\tilde{L},L) -\mathbb{E}[P])^2 \right] \leq \varepsilon^2,
\end{equation*}
and 
\begin{align*}
\text{Work}\left[\mathcal{A}_{\text{MLMC}}(\tilde{L},L) \right] &:= \sum_{l=0}^{L} K_l \text{Work} \left[\phi_{N}^{l} - \varphi_{N}^{l-1}  \right] 
= 
\begin{cases}
\mathcal{O}\left( \varepsilon^{-2 - \frac{\tilde{\gamma} - \tilde{c}}{\tilde{w}}} \right), & \text{ if } s > \gamma, \\
\mathcal{O}\left( \varepsilon^{-2 - \frac{\tilde{\gamma} - \tilde{c}}{\tilde{w}} } \log^2( \varepsilon) \right), & \text{ if } s = \gamma, \\
\mathcal{O}\left( \varepsilon^{-2 - \frac{\tilde{\gamma} - \tilde{c}}{\tilde{w}} - \frac{\gamma-s}{w}}  \right), & \text{ if } s < \gamma.
\end{cases}
\end{align*}
\end{prp}

We apply this result in the setting of the antithetic approach 
with a fixed number of particles $N = 2^{\tilde{L}}$ across all levels and $M_l = 2^l$ time-steps on level $l$ for
$l \in \lbrace 0, \ldots, L \rbrace$.

For sufficiently regular payoff functions $P$ (e.g., twice continuously differentiable with bounded first and second derivatives),
it follows as in 
Theorem \ref{TheoremAntithetic} that $\mathbb{E} [|\phi^l_N - \varphi^{l-1}_N|^2] \leq C \delta_l^{2}$
(see also \cite[Lemma 2.2]{MGLS}).
We now assume instead the stronger, multiplicative bound
\begin{align}
\label{ass_antith}
\mathbb{E} [|\phi^l_N - \varphi^{l-1}_N|^2] \leq C \frac{\delta_l^{2}}{N}.
\end{align}
This bound is consistent with our numerical tests (see Figure \ref{fig:StrongOrder1Antithetic2}, top right) and is also the one assumed in \cite{AT}, but the proof 
(in the present context) is elusive so far.

Then, in the setting of Proposition \ref{TH:MLMCCali}, $s=2$ and $\gamma=1$, assuming $P$ to be regular enough.
Hence, for a given $\varepsilon>0$,
from $w=1$ we choose the number of levels $L = \mathcal{O}(\log(\varepsilon^{-1})/\log(2))$, and
assuming $\tilde{w}=1$, we choose
the number of particles across all levels $N = 2^L$ (i.e., $\tilde{L} = L$).

Following \cite{LSAT}, we denote now by $p=0,1$ the order of interaction in the particle system, i.e., the cost required to compute all interaction terms is of order $N^{p+1}$, and $\tilde{\gamma}=p+1$, $\gamma = 1$ above.
The proof of Theorem 3.1 in \cite{AT} then reveals that the optimal number of particle systems per level is $K_l = \mathcal{O}(\varepsilon^{-1}  2^{-3l/2})$.
Note that the implied constants do not depend on $l$.

From Proposition \ref{TH:MLMCCali}, we hence deduce the overall computational cost as follows. 
\begin{cor}\label{cor2}
Under the assumption of \eqref{ass_antith},
the optimal complexity of the antithetic multilevel estimator for MSE $\varepsilon^2$ is of order $\varepsilon^{-2-p}$, for interaction terms of order $p$.
\end{cor}

For standard MLMC without the antithetic technique (i.e., $s=1$),
setting $K_l = \mathcal{O}(\varepsilon^{-1}(L+1)2^{-l})$ (with the same choices of $L$, $M_l$, and $N$), we derive that the overall cost is of order $\varepsilon^{-2-p} \log^2(\varepsilon)$. A plain MC approach gives order $\varepsilon^{-3-p}$, as $N,K$ and $M$ need to be chosen of order $\mathcal{O}(\varepsilon^{-1})$ to obtain the desired accuracy.

Proposition \ref{TH:MLMCCali} contrasts with the classical MLMC setting of \cite{MG} in that the multilevel decomposition of the estimator only acts in the index related to $l$, but not the one related to $\tilde{L}$. That is to say, in our setting, we vary the number of time-steps $M$ across levels, but not the number of particles $N$.
This leads to optimal complexity because the variance is assumed to decay in both $M$ and $N$ due to \eqref{ass_antith}.

In the setting without delay and for a constant diffusion coefficient, \cite{LSAT} proposes an antithetic scheme with respect to the number of particles -- as opposed to antithetic in the discrete-time paths as we do here -- and proves an additive variance bound of the form
$\mathcal{O}\left(\delta^{2} + 1/N^2 \right)$.
The order $1/N^2$ is an improvement over the order $1/N$ which is expected for propagation of chaos without antithetic sampling, while the order $\delta^{2}$
holds for the Euler--Maruyama scheme because it coincides with the Milstein scheme for constant diffusion coefficient.
This allows a similar complexity analysis as presented above; in particular, \cite[Theorem 6.3]{LSAT} shows how to choose a sequence $K_l$ in order to obtain a result  analogous to Corollary \ref{cor2}. 

To obtain a combined MLMC estimator, i.e., employing the antithetic approach proposed in our paper together with the antithetic scheme with respect to the number of particles, we would choose $M_l = N_l = 2^l$ and define {\color{black}
\begin{equation*}
\phi_{N_l}^{l} := \frac{1}{N_l} \sum_{j=1}^{N_l} \tfrac{P(Y^{j,N_l,f}(T)) + P(Y^{j,N_l,a}(T))}{2}, \qquad
\varphi_{N_l}^{l-1} := \frac{1}{N_l} \sum_{j=1}^{N_l} \tfrac{P(Y^{j,N_l,(1),c}(T)) + P(Y^{j,N_l,(2),c}(T))}{2}, 
\end{equation*}}
where $Y^{j,N_l,(1),c}(T)$ is a particle system of size $N_l/2$ and uses the first $N_l/2$ Brownian motions and initial data from the set $(W^{i},\xi^{i})_{i \in \lbrace 1, \ldots, N_l \rbrace}$, while $Y^{j,N_l,(2),c}(T)$ is the system associated with the other half. 

As commented above, such an estimator is expected to give a variance decay $\mathcal{O}\left(2^{-2 l} \right)$.
The optimal complexity under this assumption 
is the same as derived in Corollary \ref{cor2}. 

\section{Numerical results}\label{Section:Sec4}

We now present a number of numerical tests to illustrate the practical behaviour of the schemes proposed in this article. 
We use the canonical particle approximation \cite{BT} to the law $\mathcal{L}_{Y_{t_n}}$ at each time-step $t_n$ by its empirical distribution. For our numerical experiments, we used $N=10^3$, unless stated otherwise.

As we do not know the exact solution in the considered examples, the convergence rates with respect to the number of time-steps were determined by comparing two solutions (at time $T=1$) computed on a fine and coarse time grid, respectively, 
where the same Brownian increments were used for both. 
In order to assess the strong convergence in $\delta$, we thus compute the root-mean-square error (RMSE) 
\begin{equation*}
\text{RMSE}:= \sqrt{\frac{1}{N} \sum_{j=1}^{N} \left(Y^{j,N,l}(T) - Y^{j,N,l-1}(T) \right)^2},
\end{equation*}
where $Y^{j,N,l}(T)$ denotes the approximation of $X$ at time $T$ computed with $N$ particles and $2^lT$ time-steps.

For simplicity, we assume that $s_1, \ldots, s_k$ are contained in the considered time-grid, i.e., that they are of the form $-n\delta_{l_0}$ ({\color{black}where $l_0$ refers to the coarsest levels}) for some {\color{black}non-negative} integer $n$.

\subsection{First order strong convergence of the Milstein scheme}

{\color{black}This section numerically illustrates the convergence of the tamed Milstein schemes for point-delay McKean--Vlasov SDEs for two test cases. Specifically, we compare Schemes 1 and 2 from Section \ref{Section:Sec2} in Fig.\ \ref{PointDelay1}, top left.}
We also give implementation details of the scheme including the computation of the iterated stochastic integrals.

\textbf{Example 1}: 
Here, we consider the point-delay McKean--Vlasov SDE
\begin{align*}
\mathrm{d}X(t) &= \left( 1 -(X(t))^3 + X(t) + \frac{1}{k}\sum_{l=1}^{k} X(t + s_l)  + \frac{1}{k}\sum_{l=1}^{k}  \mathbb{E}[X(t + s_l)]  \right) \, \mathrm{d}t \\
& \quad + \frac{1}{k}\sum_{l=1}^{k} X(t + s_l) \, \mathrm{d}W(t), \quad t \in [0,T],
\end{align*}
with the initial value $X_0 = \xi = \boldsymbol{0} \in \mathscr{C}$, $s_1, \ldots, s_k \in [-\tau,0]$. Hence, the tamed Milstein approximation for this example results in the particle system 
\begin{align*}
Y^{i,N}(t_{n+1}) &= Y^{i,N}(t_n)  +  b_{\delta}(\Pi(Y_{t_n}^{i,N}),\Pi(\mu_{t_n}^{Y,N})) \delta  + \sigma( \Pi(Y_{t_n}^{i,N}),\Pi(\mu_{t_n}^{Y,N}))\Delta W_n^{i} \\
& \quad + \sum_{l=1}^{k} \partial_{x_l} \sigma( \Pi(Y_{t_n}^{i,N}),\Pi(\mu_{t_n}^{Y,N})) \tilde{\sigma}({\color{black}t_n + s_l},\Pi(Y_{t_n + s_l}^{i,N}),\Pi(\mu_{t_n}^{Y,N})) I^{i}(t_n + s_l,t_{n+1} +s_l; s_l),
\end{align*}  
where the driving Brownian motions $W^{i}$, $i \in \mathbb{S}_N$, are independent, and 
\begin{align*}
& b(\Pi(Y_{t_n}^{i,N}),\Pi(\mu_{t_n}^{Y,N})) = 1 - (Y^{i,N}(t_n))^3 + Y^{i,N}(t_n)  + \frac{1}{k}\sum_{l =1}^{k} Y^{i,N}(t_n + s_l) + \frac{1}{N} \sum_{j=1}^{N} \frac{1}{k} \sum_{l =1}^{k} Y^{j,N}(t_n + s_l),  \\
& \sigma(\Pi(Y_{t_n}^{i,N}),\Pi(\mu_{t_n}^{Y,N})) = \frac{1}{k}\sum_{l =1}^{k} Y^{i,N}(t_n + s_l).
\end{align*}
In addition, for Scheme 1 we have
\begin{equation*}
b_{\delta}( \Pi(Y_{t_n}^{i,N}),\Pi(\mu_{t_n}^{Y,N})) = \frac{ b(\Pi(Y_{t_n}^{i,N}),\Pi(\mu_{t_n}^{Y,N}))}{1 + \delta |b( \Pi(Y_{t_n}^{i,N}),\Pi(\mu_{t_n}^{Y,N}))|},
\end{equation*}
{\color{black}and with the additional exponent $\bar{q}$ for Scheme 2.}
Further, recall that $\partial_{x_l}$ denotes the derivative with respect to the $l$-th state component.

{\color{black}Above, $I^{i}(t_n + s_l,t_{n+1} +s_l; s_l)$ denotes the iterated stochastic integrals, with are approximated using a truncation parameter $r=M^{1/2}$ as described in Section \ref{subsec:milstein}.}
 
In the test, the delay parameters were $\tau = 1/8$ and $k=2$ (i.e., $s_1=0$ and $s_2 = -\tau$). 
The strong convergence of the discretised particle system is depicted in Fig.\ \ref{PointDelay1}, top left, where we observe the expected order one. 

\textbf{Example 2}:
In analogy to above, we additionally study the example
\begin{align*}
\mathrm{d}X(t) &= \left( 1 -(X(t))^3 + X(t) + \frac{1}{k}\sum_{l=1}^{k} X(t + s_l)  + \frac{1}{k}\sum_{l=1}^{k}  \mathbb{E}[X(t + s_l)]  \right) \, \mathrm{d}t \\
& \quad + \left( \frac{1}{k}\sum_{l=1}^{k}  \mathbb{E}[X(t + s_l)] \right) \, \mathrm{d}W(t), \quad t \in [0,T],
\end{align*}
with the initial value $X_0 = \xi = \boldsymbol{0} \in \mathscr{C}$, $s_1, \ldots, s_k \in [-\tau,0]$. 
As the diffusion coefficient does not depend on the state, the Milstein scheme reduces to the Euler--Maruyama scheme.
We confirm numerically in Fig.\ \ref{PointDelay1}, top left,
that the tamed Euler--Maruyama scheme converges in this case with order one.
\begin{figure}[!h]
\centering
\includegraphics[width=0.48\textwidth]{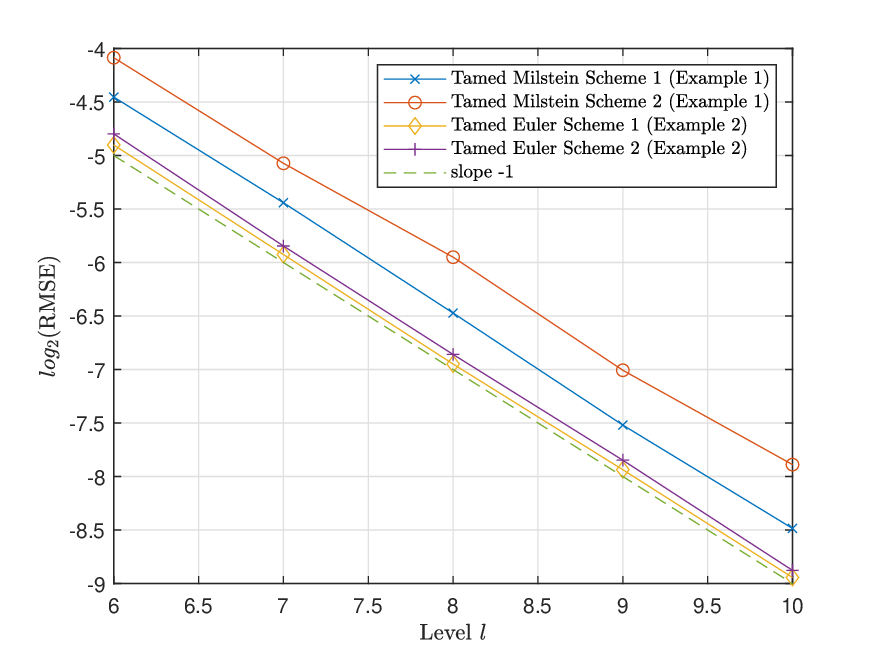} \hfill
\includegraphics[width=0.48\textwidth]{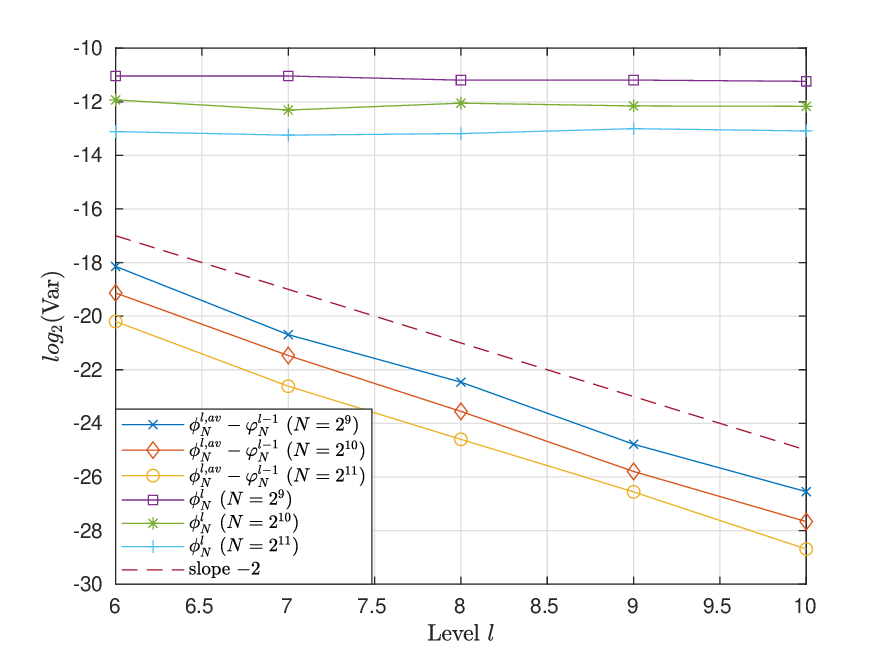} \hfill
 \includegraphics[width=0.48\textwidth]{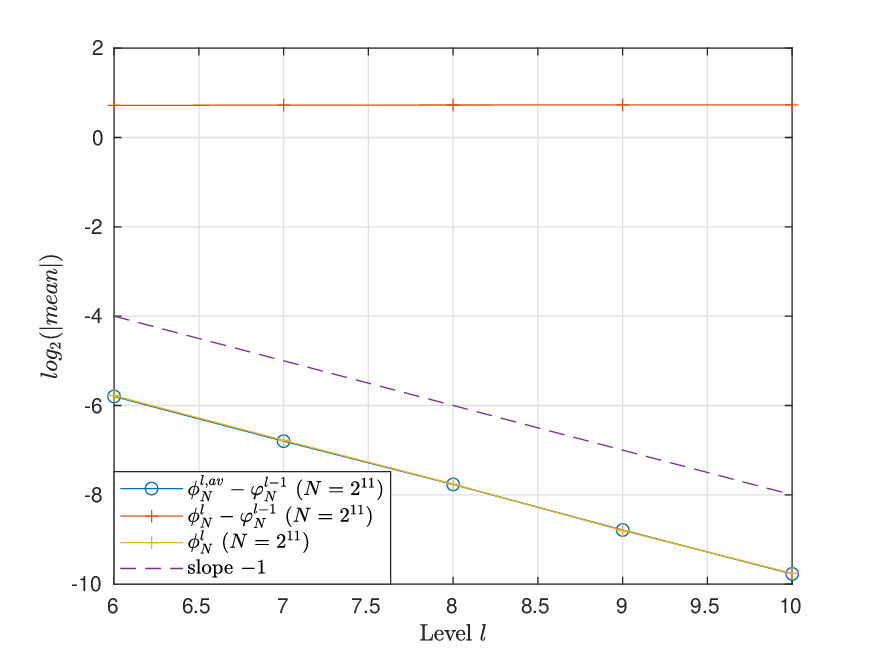} \hfill
 \includegraphics[width=0.48\textwidth]{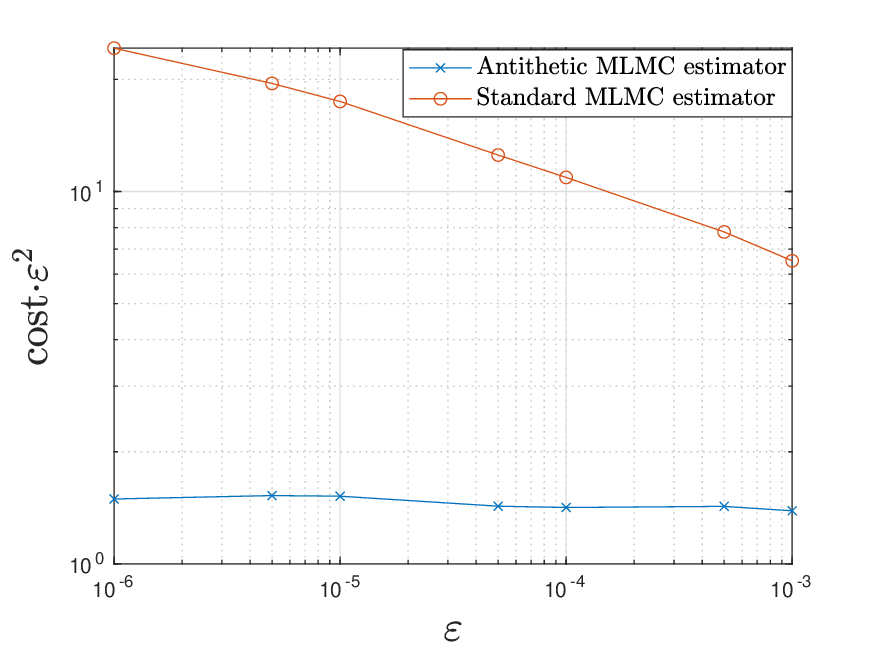}
\caption{Left top: Strong convergence of the tamed 
{\color{black} Milstein schemes (Scheme 1 and 2) for Example 1 and
the tamed Euler--Maruyama schemes for Example 2 (equivalent to Milstein in this case).
The following are for Scheme 2.}
Right top: Variance decay of the multilevel correction terms using the antithetic approach compared to standard MC. Left bottom: 
Decay of the expected value of the multilevel correction terms using the antithetic approach compared to standard MC.  Right bottom: Computational complexity.}
\label{PointDelay1}
\label{fig:StrongOrder1Antithetic}
\label{fig:StrongOrder1Antithetic2}
\label{fig:AntitheticMean}
\end{figure}

\subsection{Convergence and complexity of the antithetic multilevel sampling scheme}\label{SEC:TR}

Here, we numerically illustrate the performance of the MLMC Milstein method which uses the antithetic approach presented in Section \ref{Section:SectionAnti} applied to Example 1 from above. 
{\color{black} We focus on Scheme 2, which was found to be slightly more accurate numerically in the previous section.}
In particular, we will demonstrate the variance decay of the multilevel correction terms described in equation \eqref{eq:VarAssump} and investigate the overall computational complexity to estimate $\mathbb{E}[P(X_T)]$ with given accuracy $\varepsilon>0$ when the estimator \eqref{eqn:MLML} is used.  

Using the same notation as in Section \ref{Section:Sec2}, the scheme has the form 
\begin{align*}
Y^{i,N}(t_{n+1}) &=  Y^{i,N}(t_n) + b_{\delta}(\Pi(Y^{i,N}_{t_n}),\Pi(\mu_{t_n}^{Y,N}))\delta + \sigma(\Pi(Y^{i,N}_{t_n}),\Pi(\mu_{t_n}^{Y,N}))\Delta W^{i}_n \\
& \quad + \frac{1}{2} \sum_{l=1}^{k} \partial_{x_l} \sigma(\Pi(Y^{i,N}_{t_n}),\Pi(\mu_{t_n}^{Y,N})) \tilde{\sigma}({\color{black}t_n + s_l},\Pi(Y_{t_n + s_l}),\Pi(\mu_{t_n+s_l}^{Y,N}))\left( \bar{W}^{i}(\delta)\bar{B}^{i}(\delta) - \mathrm{I}_{\lbrace s_l=0 \rbrace} \delta \right).
\end{align*} 
As payoff function, we choose for simplicity $P(x)=x$, but
a similar behaviour is expected for any smooth enough $P$. 

Now, on a fine time-grid, we compute two approximate solutions to the particle system at time $T$, denoted by $Y^{i,N,f}(T)$ and $Y^{i,N,a}(T)$, where the simulation of $Y^{i,N,a}(T)$ is based on the antithetic Brownian paths of the paths used to compute $Y^{i,N,f}(T)$, i.e., the odd and even Brownian increments are interchanged. We then determine the variance \eqref{eq:VarAssump}, where $\phi_{N}^{l,av}$ and $\varphi_{N}^{l-1}$ are estimated using the average $(Y^{i,N,f}(T) + Y^{i,N,a}(T))/2$ and $Y^{i,N,c}(T)$, a coarse path solution, respectively. 
The quantity $\phi_{N}^{l,av}$ is estimated using the samples $Y^{i,N,f}(T)$ only. 
We choose $N \in \lbrace 2^{9}, 2^{10}, 2^{11} \rbrace$, to numerically confirm the additional factor $1/N$ in \eqref{eq:VarAssump}. 
 
The computations are based on the choices $k=2$ and $\tau=1/8$ using the SDE in Example 1 given in the previous section. 
Fig.\ \ref{fig:StrongOrder1Antithetic} shows that the strong convergence rate of the proposed scheme is one (i.e., $s=2$). The variance of the standard MC estimator barely varies with the levels.
Additionally, in Fig.\ \ref{fig:AntitheticMean} we depict the decay of the the expected value of the MLMC correction terms. We observe that this decay is of order one. As anticipated, the antithetic MLMC and standard MLMC estimator have the same
expected value.

To investigate the required complexity of the proposed MLMC estimator to achieve an accuracy $\varepsilon >0$ (in a RMSE sense), we plot $\varepsilon^{-2} \cdot \textrm{cost}$ against the accuracy $\varepsilon$, see Fig.\ \ref{fig:StrongOrder1Antithetic2}. Here, $\textrm{cost}$ denotes the complexity to compute (\ref{eqn:MLML}) with given $L$, $K_l$, $M_l= 2^l$ and $N$, i.e., $ \textrm{cost} = \mathcal{O}(N \sum_{l=0}^{L} K_l M_l)$. Since, in Example 1, the law-dependence is given by an expectation, the empirical distribution only has to be computed once (at cost $N$) at each time-step, so $p=0$ in Section \ref{sec:complex}.

\section{Proof of Theorem \ref{THDelay:THDelay4}} \label{Section:Sec3}

Here, and throughout the remaining article, we write $a \1 b$ to express that there exists a constant $C>0$ such that $a \leq Cb$, where $a,b \in \mathbb{R}$. We remark that the implied constant $C>0$ might be dependent on $p,\vv,T, m, d$ i.e.,
$C=C(\vv,p,T,m,d)$, but is independent of $M$ and $N$; also, they may change their values from line to line in a sequence of inequalities. 
We restrict the theoretical investigation to Scheme 2, as proving a moment bound for Scheme 1 is very involved and requires a lengthy (non-standard) analysis. 

\subsection{Proof of Theorem \ref{THDelay:THDelay4}}
For the proof of Theorem \ref{THDelay:THDelay4}, we will focus on the key differences to the non-delay case (see \cite{stock}) and omit other details. We again consider the one point-delay case for most parts of the subsequent discussion. 
We define, for $t \geq 0$ and $i \in\mathbb{S}_N$,
\begin{align}\label{Delay1}
\Psi_t^{i,k} &:=  \sum_{l=1}^{k} \partial_{x_l} \si(\Pi(Y_{t_\dd }^{i,N}), \Pi(\mu_{t_\dd}^{Y,N})) \tilde{\si}(\Pi(Y_{t_\dd +s_l}^{i,N}),\Pi(\mu_{t_\dd + s_l}^{Y,N})) \int_{t_\dd +s_l}^{t + s_l} \d W^i(r) \nonumber \\
& \qquad + \sum_{l=1}^{k} \ff{1}{N}\sum_{j=1}^N   \partial_{\mu_l} \si(\Pi(Y_{t_\dd}^{i,N}),\Pi(\mu_{t_\dd}^{Y,N}))(Y^{j,N}(t_\dd+s_l))\tilde{\si}(t_{\dd} + s_l,\Pi(Y_{t_\dd +s_l}^{i,N}),\Pi(\mu_{t_\dd + s_l}^{Y,N})) \int_{t_\dd + s_l}^{t+s_l} \d W^j(r), \nonumber \\
\Upsilon_t^{i,k} &:=
\si(\Pi(Y_{t_\dd}^{i,N}),\Pi(\mu_{t_\dd}^{Y,N})) + \Psi_t^{i,k}, \nonumber \\
 \Gamma_t^{i,k} &:=
\si(\Pi(Y_t^{i,N}), \Pi(\mu_t^{Y,N}))-\Upsilon_t^{i,k} \nonumber \\
M_t^{i,k} &:=\int_0^t (\si(\Pi(X_s^{i,N}),\Pi(\mu_s^{X,N}))-\Upsilon_s^{i,k}) \, \d W^i(s). 
\end{align}
\begin{lem}\label{TH:TH3}
Let $Y^{i,N}(t_n)$ for $n \in \lbrace 0, \ldots, M \rbrace$, be defined as in (\ref{eq:DelayParticleSystem}). Then, under Assumptions ({\bf AD$_b^1$}), ({\bf AD$_\sigma^1$})--({\bf AD$_\sigma^2$}) and ({\bf H$_1$}), for any $p \geq 2$ there exists a constant $C >0$ (independent of $N$ and $M$) such that
\begin{equation*} 
   \max_{i \in \mathbb{S}_N} \max_{n \in \lbrace 0, \ldots, M \rbrace} \mathbb{E}[| Y^{i,N}(t_n) |^p]  \leq C. 
\end{equation*} 
\end{lem} 
\begin{proof}
The proof can be carried out using It\^{o}'s formula along with a standard Gronwall-type argument (see, e.g., \cite[Lemma 4]{K}) in combination with an inductive procedure outlined in the proof of Proposition \ref{Prop:MomentDelayMcKean}. We note that item (a) in ({\bf AD$_\sigma^2$}) on the uniform boundedness of derivatives of $\sigma$ with respect to the delay variables implies that
\begin{align*}
&\mathbb{E} \left[\left|  \partial_{x_2} \sigma(\Pi(Y^{i,N}_{t_\dd}),\Pi(\mu_{t_\dd}^{Y,N})) \tilde{\sigma}({\color{black}t_\dd + s_2},\Pi(Y^{i,N}_{t_\dd + s_2}),\Pi(\mu_{t_\dd +s_2}^{Y,N})) \int_{t_\dd +s_2}^{t+ s_2}  \mathrm{d} W^{i}(r) \right|^p \right]  \\
&\1 \mathbb{E} \left[ \left| \int_{t_\dd +s_2}^{t+ s_2} \tilde{\sigma}({\color{black}t_\dd + s_2},\Pi(Y^{i,N}_{t_\dd + s_2}),\Pi(\mu_{t_\dd +s_2}^{Y,N}))  \, \mathrm{d} W^{i}(r) \right|^p \right] \\
&\1 \delta^{p/2} \left(1+ \mathbb{E} \left[|Y^{i,N}({t_\dd -\tau })|^p \right] + \mathbb{E} \left[|Y^{i,N}({t_\dd -2\tau })|^p \right]\right), 
\end{align*} 
where we used the growth assumptions on $\sigma$ ({\bf AD$_\sigma^1$}) with $q_2=0$, and {\color{black}for simplicity consider $k=2$ and $s_2 = -\tau$}. 

{\color{black}Observe that the uniform boundedness of derivatives with respect to the delay variables is crucial, as in the term
\begin{equation*}
\partial_{x_2} \sigma(\Pi(Y^{i,N}_{t_\dd}),\Pi(\mu_{t_\dd}^{Y,N})) \int_{t_\dd +s_2}^{t+ s_2} \mathrm{d} W^{i}(r), 
\end{equation*} 
$\partial_{x_2} \sigma(\Pi(Y^{i,N}_{t_\dd}),\Pi(\mu_{t_\dd}^{Y,N}))$ is anticipative.

However, if the second assumption item (b) in ({\bf AD$_\sigma^2$}) holds, then it becomes non-anticipative. The growth assumption on $\partial_{x_2}\sigma_u \sigma$ in ({\bf AD$_\sigma^2$}) gives
\begin{align*}
&\mathbb{E} \left[ \left|  \partial_{x_2} \sigma(\Pi(Y^{i,N}_{t_\dd}),\Pi(\mu_{t_\dd}^{Y,N})) \tilde{\sigma}({\color{black}t_\dd + s_2},\Pi(Y^{i,N}_{t_\dd + s_2}),\Pi(\mu_{t_\dd +s_2}^{Y,N})) \int_{t_\dd +s_2}^{t+ s_2}  \mathrm{d} W^{i}(r) \right|^p \right]  \\
&\1 \delta^{p/2} \left(1+ \mathbb{E}[|Y^{i,N}({t_\dd -\tau })|^{p(1 + q_3)}] + \mathbb{E}[|Y^{i,N}({t_\dd -2\tau })|^{p(1+q_3)}] + \mathbb{E}[|Y^{i,N}({t_\dd -\tau })|^p] + \mathbb{E}[|Y^{i,N}({t_\dd -2\tau })|^p] \right).
\end{align*}}
\end{proof}
{\color{black}
\begin{remark}
It seems that techniques from anticipative calculus, see \cite{Nualart}, are necessary if one wishes to allow (unbounded) mixed terms involving delay and non-delay variables in the diffusion coefficient.
\end{remark}
}
{\color{black}In the sequel, we set $Z^{i,N}:=X^{i,N}-Y^{i,N}$, where $(X^{i,N})_{i \in \mathbb{S}_N}$ is the particle system defined in \eqref{eq:DelayParticleSystem_cont} and $(Y^{i,N})_{i \in \mathbb{S}_N}$ is defined by (\ref{eq:DelayParticleSystem}).}
\begin{lem}\label{Lemm1D}
Let Assumptions ({\bf AD}$_b^1$), ({\bf AD}$_\si^1$)--({\bf AD}$_\si^3$), ({\bf H$_1$}) hold, and $k=2$ {(\color{black}for simplicity)}. Then, for any $p \geq 2$ there is a constant $C>0$ such that, for all $t \in [0,T]$
\begin{align}\label{F1D}
\Lambda_t^{i,2,p} &:=\int_0^t(\E [\|\si(\Pi(X_s^{i,N}),\Pi(\mu_s^{X,N
}))-\Upsilon_s^{i,2}\|^{2p}])^{\ff{1}{p}} \d s  \nonumber \\
& \quad \le C\Big\{\dd^2+\int_0^t (\E[|Z^{i,N}(s)|^{2p}])^{\ff{1}{p}}\d
s \nonumber \\
& \hspace{1.65cm} + {\color{black} \int_0^{0 \lor (t-\tau)}  (\E [|Z^{i,N}(s)|^{2p}(1+|X^{i,N}(s)|^{2pq_2} + |Y^{i,N}(s)|^{2pq_2})])^{\ff{1}{p}}\, \d
s} \Big \}.
\end{align}
\end{lem}
\begin{proof}
From ({\bf AD}$_\si^1$) (possibly with $q_2 =0$)  and Minkowski's inequality,
we derive that
\begin{equation*}
\begin{split}
\Lambda_t^{i,2,p} &\1\int_0^t (\E [\|\si(\Pi(X_s^{i,N}),\Pi(\mu_s^{X,N}
))-\si(\Pi(Y_s^{i,N}),\Pi(\mu_s^{Y,N})) \|^{2p}])^{\ff{1}{p}} \,\d
s+\int_0^t (\E[\|\Gamma_s^{i,2}\|^{2p}])^{\ff{1}{p}}\, \d s\\
&\1\int_0^t  (\E[|Z^{i,N}(s)|^{2p}])^{\ff{1}{p}}\, \d
s+\ff{1}{N}\sum_{j=1}^N\int_0^t(\E[|Z^{j,N}(s)|^{2p}])^{\ff{1}{p}}\, \d s \\
& \hspace{1cm} + {\color{black}\int_0^{0 \lor (t-\tau)}  (\E [|Z^{i,N}(s)|^{2p}(1+|X^{i,N}(s)|^{2pq_2} + |Y^{i,N}(s)|^{2pq_2})])^{\ff{1}{p}}\, \d
s} \\
& \hspace{1cm} + \ff{1}{N}\sum_{j=1}^N\int_0^{0 \lor (t-\tau)} (\E [|Z^{j,N}(s)|^{2p}] )^{\ff{1}{p}}\, \d s +  \int_0^t(\E [\|\Gamma_s^{i,k}\|^{2p}])^{\ff{1}{p}}\, \d s\\
&\1\int_0^t (\E[|Z^{i,N}(s)|^{2p}])^{\ff{1}{p}}\, \d s +\int_0^t(\E [\|\Gamma_s^{i,k}\|^{2p}])^{\ff{1}{p}}\, \d s \\
& \qquad + {\color{black} \int_0^{0 \lor (t-\tau)} (\E [|Z^{i,N}(s)|^{2p}(1+|X^{i,N}(s)|^{2pq_2} + |Y^{i,N}(s)|^{2pq_2})])^{\ff{1}{p}}\, \d s,}
\end{split}
\end{equation*}
where in the last display we used the fact that
$Z^{j,N},j\in\mathbb{S}_N,$ are identically distributed.
Consequently, to derive \eqref{F1D}, it is sufficient to show that for $t \in [0,T]$,
\begin{equation}\label{F22D}
(\E[\|\Gamma_t^{i,2}\|^{2p}])^{\ff{1}{p}} \1 \dd^2,
\end{equation}
In the sequel, we will restrict the discussion to $d=m=1$ for ease of notation. 
First, we write 
\begin{align*}
& \sigma(Y^{i,N}(t),Y^{i,N}(t-\tau), \mu_{t}^{Y,N},\mu_{t-\tau}^{Y,N}) - \sigma(Y^{i,N}(t_\dd),Y^{i,N}(t_\dd-\tau), \mu_{t_\dd}^{Y,N},\mu_{t_\dd-\tau}^{Y,N}) \\
& = \sigma(Y^{i,N}(t),Y^{i,N}(t-\tau), \mu_{t}^{Y,N},\mu_{t-\tau}^{Y,N}) - \sigma(Y^{i,N}(t_\dd),Y^{i,N}(t-\tau), \mu_{t}^{Y,N},\mu_{t-\tau}^{Y,N}) \\ 
& \quad + \sigma(Y^{i,N}(t_\dd),Y^{i,N}(t-\tau), \mu_{t}^{Y,N},\mu_{t-\tau}^{Y,N}) - \sigma(Y^{i,N}(t_\dd),Y^{i,N}(t_\dd -\tau), \mu_{t}^{Y,N},\mu_{t-\tau}^{Y,N}) \\
& \quad + \sigma(Y^{i,N}(t_\dd),Y^{i,N}(t_\dd -\tau), \mu_{t}^{Y,N},\mu_{t-\tau}^{Y,N}) - \sigma(Y^{i,N}(t_\dd),Y^{i,N}(t_\dd -\tau), \mu_{t}^{Y,N},\mu_{t_\dd-\tau}^{Y,N}) \\
& \quad + \sigma(Y^{i,N}(t_\dd),Y^{i,N}(t_\dd -\tau), \mu_{t}^{Y,N},\mu_{t_\dd-\tau}^{Y,N}) - \sigma(Y^{i,N}(t_\dd),Y^{i,N}(t_\dd -\tau), \mu_{t_\dd}^{Y,N},\mu_{t_\dd-\tau}^{Y,N}) \\
& =: \sum_{i=1}^{4} \Xi_i.
\end{align*}
Next, using the definition $\Delta Y^{i,N}(t):= Y^{i,N}(t) - Y^{i,N}(t_\dd)$, we observe that
\begin{align*}
\Xi_1 &= \int_{0}^{1} \frac{\d}{\d \lambda} \sigma(Y^{i,N}(t_\dd) + \lambda \Delta Y^{i,N}(t),Y^{i,N}(t-\tau), \mu_{t}^{Y,N},\mu_{t-\tau}^{Y,N}) \, \d \lambda \\
&=  \int_{0}^{1} \partial_{x_1} \sigma(Y^{i,N}(t_\dd) + \lambda \Delta Y^{i,N}(t), Y^{i,N}(t-\tau),\mu_{t}^{Y,N},\mu_{t-\tau}^{Y,N})\Delta Y^{i,N}(t) \, \d \lambda.
\end{align*}
Therefore, simple computations show, due to assumption ({\bf AD}$_\si^3$)
\begin{align}\label{eqH}
&\mathbb{E} \Big[\Big| \Xi_1 -  \partial_{x_1} \sigma(Y^{i,N}(t_\dd), Y^{i,N}(t_\dd -\tau),\mu_{t_\dd}^{Y,N},\mu_{t_\dd-\tau}^{Y,N}) \Delta Y^{i,N}(t)  \Big|^2 \Big] \nonumber \\
&= \mathbb{E} \Big[\Big| \int_{0}^{1} \big( \partial_{x_1} \sigma(Y^{i,N}(t_\dd) + \lambda \Delta Y^{i,N}(t), Y^{i,N}(t-\tau), \mu_{t}^{Y,N},\mu_{t-\tau}^{Y,N}) \nonumber \\
& \qquad -\partial_{x_1} \sigma(Y^{i,N}(t_\dd), Y^{i,N}(t_\dd -\tau),\mu_{t_\dd}^{Y,N},\mu_{t_\dd-\tau}^{Y,N}) \big) \Delta Y^{i,N}(t) \, \d \lambda \Big|^2 \Big] \nonumber \\
& \1  \mathbb{E}[|Y^{i,N}(t) - Y^{i,N}(t_\dd)|^4] + \left(\mathbb{E}[|Y^{i,N}(t) - Y^{i,N}(t_\dd)|^4] \right)^{1/2} \left( \mathbb{E}[|Y^{i,N}(t-\tau) - Y^{i,N}(t_\dd -\tau)|^4] \right)^{1/2} \nonumber \\
& \qquad + \left( \frac{1}{N} \sum_{j=1}^{N} \mathbb{E}[|Y^{j,N}(t) - Y^{j,N}(t_\dd)|^4] \right)^{1/2} \left( \mathbb{E}[|Y^{i,N}(t) - Y^{i,N}(t_\dd)|^4] \right)^{1/2} \1 \delta^2,
\end{align}
where we used ({\bf H$_1$}) and $\mathbb{E}[|Y^{i,N}(t) - Y^{i,N}(t_\dd)|^4] \1 \delta^2$. {\color{black}The claim regarding the one-step error is a consequence of the identity 
\begin{align}\label{eq:one-step}
Y^{i,N}(t) - Y_{t_{\dd}}^{i,N} &= b_{\delta}(Y^{i,N}(t_\dd),\mu_{t_\dd}^{Y,N})(t-t_\dd) + \sigma({\color{black}\Pi(Y^{i,N}_{t_\dd})},\Pi(\mu_{t_\dd}^{Y,N}))\int_{t_\dd}^{t} \, \mathrm{d}W^{i}(s)  \nonumber \\
& \quad + \sum_{l=1}^{2}\partial_{x_l}\sigma(\Pi(Y^{i,N}_{t_\dd}),\Pi(\mu_{t_\dd}^{Y,N})) \tilde{\sigma}({\color{black}t_\dd + s_l},\Pi(Y^{i,N}_{t_\dd + s_l}),\Pi(\mu_{t_\dd +s_l}^{Y,N})) \int_{t_\dd}^{t} \int_{t_\dd +s_l}^{s+ s_l}  \, \mathrm{d} W^{i}(u)\, \mathrm{d} W^{i}(s)  \nonumber \\
& \quad + \sum_{l=1}^{2} \frac{1}{N}\sum_{j = 1}^{N}  \partial_{\mu_l} \sigma(\Pi(Y^{i,N}(t_\dd)), \Pi(\mu_{t_\dd}^{Y,N}))(Y^{j,N}(t_\dd +s_l))  \nonumber \\
& \hspace{2cm} \times \tilde{\sigma}({\color{black}t_\dd + s_l},\Pi(Y_{t_\dd+s_l}^{j,N}), \Pi(\mu_{t_\dd+s_l}^{Y,N}))  \ \int_{t_\dd}^{t} \int_{t_\dd +s_l}^{s+ s_l}  \, \mathrm{d} W^{j}(u)\, \mathrm{d} W^{i}(s),
\end{align}
and standard inequalities along with the moment stability of $Y^{i,N}$.}
A similar estimate holds if $\Xi_1$ is replaced by $\Xi_2$. 

Using \cite[Proposition 5.35]{CD} we get that
\begin{align*}
\Xi_4 &= \int_{0}^{1} \frac{\d}{\d \lambda} \sigma(Y^{i,N}(t_\dd),Y^{i,N}(t_\dd -\tau), \mu_{t}^{\lambda, Y,N},\mu_{t_\dd-\tau}^{Y,N}) \, \d \lambda \\
&=  \int_{0}^{1} \frac{1}{N}\sum_{j=1}^{N}  \partial_{\mu_1} \sigma(Y^{i,N}(t_\dd),Y^{i,N}(t_\dd -\tau),\mu_{t}^{\lambda, Y,N},\mu_{t_\dd-\tau}^{Y,N})(Y^{j,N}(t_\dd) + \lambda \Delta Y^{j,N}(t) )\Delta Y^{j,N}(t) \, \d \lambda,
\end{align*}
with the definition
\begin{equation*}
\mu_{t}^{\lambda, Y,N}(\mathrm{d}x) := \frac{1}{N} \sum_{j=1}^{N} \delta_{Y^{j,N}(t_{\delta}) + \lambda \Delta Y^{j,N}(t) }(\mathrm{d}x).
\end{equation*}
By virtue of ({\bf AD}$_\si^3$) and employing similar estimates as in (\ref{eqH}), we derive
\begin{align}\label{eqG}
& \mathbb{E} \Big[ \Big| \Xi_4 - \frac{1}{N}\sum_{j=1}^{N}  \partial_{\mu_1} \sigma(Y^{i,N}(t_\dd),Y^{i,N}(t_\dd -\tau) ,\mu_{t_\dd}^{Y,N},\mu_{t_{\dd}-\tau}^{Y,N})(Y^{j,N}(t_\dd)) \Delta Y^{j,N}(t)  \Big|^2 \Big] \nonumber \\
& = \mathbb{E} \Big[ \Big| \int_{0}^{1} \frac{1}{N} \sum_{j=1}^{N}  \partial_{\mu_1} \sigma(Y^{i,N}(t_\dd),Y^{i,N}(t_\dd -\tau) ,\mu_{t}^{\lambda, Y,N},\mu_{t_\dd-\tau}^{Y,N})(Y^{j,N}(t_\dd) + \lambda \Delta Y^{j,N}(t)) \Delta Y^{j,N}(t) \, \d \lambda  \nonumber \\
& \qquad - \int_{0}^{1} \frac{1}{N}\sum_{j=1}^{N}  \partial_{\mu_1} \sigma(Y^{i,N}(t_\dd),Y^{i,N}(t_\dd -\tau) ,\mu_{t_\dd}^{Y,N},\mu_{t_{\dd}-\tau}^{Y,N})(Y^{j,N}(t_\dd))  \Delta Y^{j,N}(t)  \, \d \lambda \Big|^2 \Big] \nonumber \\
& \1 \delta^2.
\end{align}
Analogous computations hold true if $\Xi_4$ is replaced by $\Xi_3$. Above considerations motivate to express
\begin{align}
\Gamma_t^{i,k} &= \sigma(Y^{i,N}(t),Y^{i,N}(t-\tau), \mu_{t}^{Y,N},\mu_{t-\tau}^{Y,N}) - \sigma(Y^{i,N}(t_\dd),Y^{i,N}(t_\dd -\tau), \mu_{t_\dd}^{Y,N},\mu_{t_\dd-\tau}^{Y,N}) \nonumber \\
& \qquad - \sum_{l=1}^{2} \partial_{x_l}\sigma(\Pi(Y_{t_\dd }^{i,N}), \Pi(\mu_{t_\dd}^{Y,N})) \tilde{\si}(t_\dd +s_l,\Pi(Y_{t_\dd +s_l}^{i,N}),\Pi(\mu_{t_\dd + s_l}^{Y,N})) \int_{t_\dd +s_l}^{t + s_l} \d
W^i(r)  \nonumber \\
& \qquad - \sum_{l=1}^{2} \ff{1}{N}\sum_{j=1}^N   \partial_{\mu_l} \si(\Pi(Y^{i,N}(t_\dd)),\Pi(\mu_{t_\dd}^{Y,N}))(Y^{j,N}(t_\dd+s_l))\tilde{\si}(t_\dd +s_l,\Pi(Y_{t_\dd +s_l}^{j,N}),\Pi(\mu_{t_\dd + s_l}^{Y,N})) \int_{t_\dd + s_l}^{t+s_l}\d W^j(r)  \nonumber \\
& = \sigma(Y^{i,N}(t),Y^{i,N}(t-\tau), \mu_{t}^{Y,N},\mu_{t-\tau}^{Y,N}) - \sigma(Y^{i,N}(t_\dd),Y^{i,N}(t_\dd -\tau), \mu_{t_\dd}^{Y,N},\mu_{t_\dd-\tau}^{Y,N})  \nonumber \\
& \qquad - \partial_{x_1}\sigma(Y^{i,N}(t_\dd), Y^{i,N}(t_\dd -\tau),\mu_{t_\dd}^{Y,N},\mu_{t_\dd-\tau}^{Y,N}) \Delta Y^{i,N}(t)  \nonumber \\
& \qquad - \partial_{x_2}\sigma(Y^{i,N}(t_\dd), Y^{i,N}(t_\dd -\tau),\mu_{t_\dd}^{Y,N},\mu_{t_\dd-\tau}^{Y,N}) \Delta Y^{i,N}(t-\tau)   \nonumber \\
&  \qquad - \frac{1}{N}\sum_{j=1}^{N}  \partial_{\mu_1} \sigma(Y^{i,N}(t_\dd), Y^{i,N}(t_\dd -\tau), \mu_{t_\dd}^{Y,N},\mu_{t_{\dd}-\tau}^{Y,N})(Y^{j,N}(t_{\dd}))  \Delta Y^{j,N}(t)  \nonumber \\
& \qquad  - \frac{1}{N}\sum_{j=1}^{N}  \partial_{\mu_2} \sigma(Y^{i,N}(t_\dd), Y^{i,N}(t_\dd -\tau), \mu_{t_\dd}^{Y,N},\mu_{t_{\dd}-\tau}^{Y,N})(Y^{j,N}(t_{\dd}-\tau)) \Delta Y^{j,N}(t-\tau)   \nonumber \\
& \qquad - \sum_{l=1}^{2} \partial_{x_l}\sigma(\Pi(Y_{t_\dd }^{i,N}), \Pi(\mu_{t_\dd}^{Y,N})) \tilde{\si}(t_\dd +s_l,\Pi(Y_{t_\dd +s_l}^{i,N}),\Pi(\mu_{t_\dd + s_l}^{Y,N})) \int_{t_\dd +s_l}^{t + s_l} \d
W^i(r) \label{eqA} \\
& \qquad - \sum_{l=1}^{2} \ff{1}{N}\sum_{j=1}^N   \partial_{\mu_l} \si(\Pi(Y^{i,N}(t_\dd)),\Pi(\mu_{t_\dd}^{Y,N}))(Y^{j,N}(t_\dd+s_l))\tilde{\si}(t_\dd +s_l,\Pi(Y_{t_\dd +s_l}^{j,N}),\Pi(\mu_{t_\dd + s_l}^{Y,N})) \int_{t_\dd + s_l}^{t+s_l}\d W^j(r) \label{eqB} \\ 
& \qquad + \partial_{x_1}\sigma(Y^{i,N}(t_\dd), Y^{i,N}(t_\dd -\tau),\mu_{t_\dd}^{Y,N},\mu_{t_\dd-\tau}^{Y,N}) \Delta Y^{i,N}(t)  \label{eqC}  \\
& \qquad + \partial_{x_2}\sigma(Y^{i,N}(t_\dd), Y^{i,N}(t_\dd -\tau),\mu_{t_\dd}^{Y,N},\mu_{t_\dd-\tau}^{Y,N}) \Delta Y^{i,N}(t-\tau)  \label{eqD} \\
&  \qquad + \frac{1}{N}\sum_{j=1}^{N}  \partial_{\mu_1} \sigma(Y^{i,N}(t_\dd), Y^{i,N}(t_\dd -\tau), \mu_{t_\dd}^{Y,N},\mu_{t_{\dd}-\tau}^{Y,N})(Y^{j,N}(t_{\dd})) \Delta Y^{j,N}(t)  \label{eqE} \\
& \qquad  + \frac{1}{N}\sum_{j=1}^{N}  \partial_{\mu_2} \sigma(Y^{i,N}(t_\dd), Y^{i,N}(t_\dd -\tau), \mu_{t_\dd}^{Y,N},\mu_{t_{\dd}-\tau}^{Y,N})(Y^{j,N}(t_{\dd}-\tau))  \Delta Y^{j,N}(t - \tau)  \label{eqF}.
\end{align}
Notice that \eqref{eq:one-step} implies
\begin{align*}
& \partial_{x_1}\sigma(Y^{i,N}(t_\dd), Y^{i,N}(t_\dd -\tau),\mu_{t_\dd}^{Y,N},\mu_{t_\dd-\tau}^{Y,N}) \Delta Y^{i,N}(t)   \\
& = \partial_{x_1}\sigma(Y^{i,N}(t_\dd), Y^{i,N}(t_\dd -\tau),\mu_{t_\dd}^{Y,N},\mu_{t_\dd-\tau}^{Y,N}) \Bigg( b_{\delta}(Y^{i,N}(t_\dd), \mu_{t_\dd}^{Y,N})(t-t_\dd) \\
& \quad + \sigma(\Pi(Y^{i}_{t_\dd}),\Pi(\mu_{t_\dd}^{Y,N}))\int_{t_\dd}^{t} \, \mathrm{d}W^{i}(s)  \nonumber \\
& \quad + \sum_{l=1}^{2} \partial_{x_l}\sigma(\Pi(Y^{i,N}_{t_\dd}),\Pi(\mu_{t_\dd}^{Y,N})) \tilde{\sigma}({\color{black}t_\dd + s_l},\Pi(Y^{i,N}_{t_\dd + s_l}),\Pi(\mu_{t_\dd +s_l}^{Y,N})) \int_{t_\dd}^{t} \int_{t_\dd +s_l}^{s+ s_l}  \, \mathrm{d} W^{i}(u)\, \mathrm{d} W^{i}(s)  \nonumber \\
& \quad + \sum_{l=1}^{2} \frac{1}{N}\sum_{j = 1}^{N}  \partial_{\mu_l} \sigma(\Pi(Y^{i,N}(t_\dd)), \Pi(\mu_{t_\dd}^{Y,N}))(Y^{j,N}(t_\dd +s_l))  \nonumber \\
& \hspace{2cm} \times \tilde{\sigma}({\color{black}t_\dd + s_l},\Pi(Y_{t_\dd+s_l}^{j,N}), \Pi(\mu_{t_\dd+s_l}^{Y,N}))  \ \int_{t_\dd}^{t} \int_{t_\dd +s_l}^{s+ s_l}  \, \mathrm{d} W^{j}(u)\, \mathrm{d} W^{i}(s) \Bigg). 
\end{align*}
Similar expressions can be derived for equations (\ref{eqD}), (\ref{eqE}) and (\ref{eqF}) given above. From this, we observe that (\ref{eqA}) and (\ref{eqB}) cancel with non-higher order terms appearing in (\ref{eqC}), (\ref{eqD}), (\ref{eqE}) and (\ref{eqF}). The remaining terms are readily proven to be of order {\color{black}$\delta$}. Recalling the estimates (\ref{eqG}) and (\ref{eqH}) allows us to deduce the claim.  
\end{proof}
\begin{lem}\label{Lem3D} 
We define for {\color{black}$t \in [0,T]$}
\begin{equation*}
\hat\Upsilon_t^{i}:= \partial_x b(Y^{i,N}({t_\dd}),\mu_{t_\dd}^{Y,N}) \si(\Pi(Y_{t_\dd }^{i,N}),\Pi(\mu_{t_\dd }^{Y,N})) \int_{t_\dd}^{t} 
\, \mathrm{d}W^i(r).
\end{equation*}
Let Assumptions ({\bf AD}$_b^1$), ({\bf AD}$_\si^1$)--({\bf AD}$_\si^3$), ({\bf H$_1$}) hold {\color{black}(and set $k=2$ for simplicity)}. Then, for all $\vv>0$ and $p \geq 2$, there
exists $C>0$ such that
\begin{align}\label{p0D}
& \Big(\E \Big[\Big \|\int_0^{\cdot} \<Z^{i,N}(s),\hat\Upsilon_s^{i}\>\, \d
s\Big\|_{\infty,t}^p \Big] \Big)^{\ff{1}{p}} \le
\vv  ( (\E [\|Z^{i,N}\|^{2p}_{\8,t}])^{\ff{1}{p}} + {\color{black} (\E [\|Z^{i,N}\|^{2p}_{\8,t-\tau}(1+\|X^{i,N}\|^{2pq_2}_{\8,t-\tau} + \|Y^{i,N}\|^{2pq_2}_{\8,t-\tau}) }])^{\ff{1}{p}} ) \notag \\
&+ C \Big\{\int_0^t 
(\E[|Z^{i,N}(s)|^{2p}])^{\ff{1}{p}}\, \d s  
 + {\color{black} \int_0^{0 \lor (t-\tau)}  (\E [|Z^{i,N}(s)|^{2p}(1+|X^{i,N}(s)|^{2pq_2} + |Y^{i,N}(s)|^{2pq_2})])^{\ff{1}{p}}\, \d
s}    + \dd^2\Big\},
\end{align}
for $t \in [0,T]$.
The same holds if $\hat\Upsilon_t^{i}$ is replaced by 
\begin{align*}
\frac{1}{N} \sum_{j=1}^{N} \partial_{\mu} b(Y^{i,N}({t_\dd }),\mu_{t_\dd }^{Y,N}
)(Y^{j,N}(t_\dd)) \si(\Pi(Y_{t_\dd}^{j,N}),\Pi(\mu_{t_\dd}^{Y,N})) \int_{t_\dd}^{t} 
\, \mathrm{d}W^j(r).
\end{align*} 
\end{lem}
\begin{proof}
The result can be proven following the same steps as \cite[Lemma 3.2]{stock} and is therefore omitted. 
\end{proof}
\begin{lem}\label{Lem2D}
Let Assumptions ({\bf AD}$_b^1$)--({\bf AD}$_b^3$), ({\bf
AD}$_\si^1$)--({\bf AD}$_\si^3$), ({\bf H$_1$}) hold {\color{black}(and set $k=2$ for simplicity)}. Then, for all $\vv>0$ and $p \geq 2$, there
exists a constant $C>0$ such that 
\begin{equation}\label{P4D}
\begin{split}
&\int_0^t(\E[|\<Z^{i,N}(s),b(Y^{i,N}(s),\mu_s^{Y,N})-b(Y^{i,N}({s_\dd}),\mu_{s_\dd}^{Y,N})\>|^p])^{\ff{1}{p}}\, \d
s\\
&\le\vv ( \,(\E[\|Z^{i,N}\|^{2p}_{\8,t}])^{\ff{1}{p}} + {\color{black} (\E [\|Z^{i,N}\|^{2p}_{\8,t-\tau}(1+\|X^{i,N}\|^{2pq_2}_{\8,t-\tau} + \|Y^{i,N}\|^{2pq_2}_{\8,t-\tau})])^{\ff{1}{p}} })\\
& \quad + C\Big\{\int_0^t
(\E[|Z^{i,N}(s)|^{2p}])^{\ff{1}{p}}\, \d s + {\color{black} \int_0^{0 \lor (t-\tau)}  (\E [|Z^{i,N}(s)|^{2p}(1+|X^{i,N}(s)|^{2pq_2} + |Y^{i,N}(s)|^{2pq_2})])^{\ff{1}{p}}\, \d
s}   +\dd^2\Big\},
\end{split}
\end{equation}
for $t \in [0,T]$.
\end{lem}
\begin{proof}
We will restrict the subsequent discussion to $d=m=1$ for ease of notation. We may write
\begin{align}
& Z^{i,N}(s) \left( b(Y^{i,N}(s),\mu_s^{Y,N})-b(Y^{i,N}(s_\dd),\mu_{s_\dd}^{Y,N}) \right) \nonumber \\
&= Z^{i,N}(s) \left( b(Y^{i,N}(s),\mu_s^{Y,N})-b(Y^{i,N}({s_\dd}),\mu_{s_\dd}^{Y,N}) \right) \label{eqDB} \\
& \quad - Z^{i,N}(s) \partial_xb(Y^{i,N}({s_\dd}),\mu_{s_\dd}^{Y,N})  \Delta Y^{i,N}(s)  \label{eqDC} \\
& \quad - Z^{i,N}(s) \frac{1}{N} \sum_{j=1}^{N} \partial_{\mu} b(Y^{i,N}({s_\dd}),\mu_{s_\dd}^{Y,N})(Y^{j,N}(s_\dd)) \Delta Y^{j,N}(s) \label{eqDD} \\
& \quad + Z^{i,N}(s)  \partial_xb(Y^{i,N}({s_\dd}),\mu_{s_\dd}^{Y,N}) \Delta Y^{i,N}(s) \label{eqDG} \\
& \quad + Z^{i,N}(s) \frac{1}{N} \sum_{j=1}^{N} \partial_{\mu} b(Y^{i,N}({s_\dd}),\mu_{s_\dd}^{Y,N})(Y^{j,N}(s_\dd)) \Delta Y^{j,N}(s). \label{eqDH}
\end{align}
The terms (\ref{eqDB})--(\ref{eqDD}) can be estimated using the techniques employed in the proof of Lemma \ref{Lemm1D}. In particular, we have due to Young's inequality
\begin{align*}
& \int_0^t \E \Big[ \Big| Z^{i,N}(s) \Big(b(Y^{i,N}(s),\mu_s^{Y,N})-b(Y^{i,N}({s_\dd}),\mu_{s_\dd}^{Y,N}) \\
& \hspace{2.5cm} -  \partial_xb(Y^{i,N}({s_\dd}),\mu_{s_\dd}^{Y,N}) \Delta Y^{i,N}(s) \\
& \hspace{2.5cm} - \frac{1}{N} \sum_{j=1}^{N} \partial_{\mu} b(Y^{i,N}({s_\dd}),\mu_{s_\dd}^{Y,N})(Y^{j,N}(s_\dd)) \Delta Y^{j,N}(s) \Big) \Big|^p \Big] \, \d
s\\
& \leq \vv\, \E[\|Z^{i,N}\|^{2p}_{\8,t}] + C\dd^{2p}. 
\end{align*} 
The remaining two terms (\ref{eqDG}) and (\ref{eqDH}) can be estimated similar to \cite[Lemma 10]{K} in combination with Lemma \ref{Lem3D}.
\end{proof}
\begin{remark}
\label{rem:anticip}
The above proof reveals why we assumed the drift to be independent of the delay variables. In case of explicit delay dependence, we would have to investigate, for $l >1$, a term of the form 
\begin{equation*}
 Z^{i,N}(s) \partial_{x_l} b(\Pi(Y^{i,N}(s_\dd)),\Pi(\mu_{s_\dd}^{Y,N})) \left(Y^{i,N}(s+s_l) - Y^{i,N}(s_\dd+s_l)  \right).
\end{equation*}
However, the equation for the difference $Y^{i,N}(s-\tau) - Y^{i,N}(s_\dd-\tau)$ depends on the delayed Brownian increment $W^{i}(s-\tau) - W^{i}(s_\dd-\tau)$, which makes $Z^{i,N}(s)$ anticipative. This difficulty does not appear in the case of globally Lipschitz continuous coefficients. The reason is the following: In the super-linear growth setting, one applies It\^{o}'s formula to $| Z^{i,N}(s)|^2$ in order to employ the one-sided Lipschitz assumption. The strong convergence analysis for globally Lipschitz continuous coefficients does not require this and therefore the problematic term discussed in this remark does not appear.
\end{remark}
\noindent
\textbf{Proof of Theorem \ref{THDelay:THDelay4}:} 
\begin{proof}
{\color{black}We aim to show that
\begin{equation*}\label{E5}
\mathbb{E} \left[ \|Z^{i,N} \|_{\infty,T}^p \right] \1 \delta^{p},
\end{equation*}
for any $p \geq 2$. By It\^{o}'s formula together with
$Z^{i,N}(0) = 0$, it follows, for $t \geq 0$, that 
\begin{equation*}
\begin{split}
|Z^{i,N}(t)|^2&= 2\int_0^t \left \langle Z^{i,N}(s),b(X^{i,N}(s),\mu_s^{X,N
})-b_{\delta}(Y^{i,N}(s_{\delta}),\mu_{s_{\delta}}^{Y,N}) \right \rangle \, \mathrm{d}s
+2\int_0^t \left \langle Z^{i,N}(s), \mathrm{d}M_s^{i,k} \right \rangle \\
&\quad+\int_0^t\|\sigma(\Pi(X^{i,N}_s),\Pi(\mu_s^{X,N}))-\Upsilon_s^{i,k}\|^2 \, \mathrm{d}s,
\end{split}
\end{equation*}
where $M^{i,k}_t$ and $\Upsilon^{i,k}_t$ for $t \in [0,T]$ are defined in \eqref{Delay1}.
The first term from above will be estimated based on the decomposition
\begin{equation*}
\begin{split}
\int_0^t \left \langle Z^{i,N}(s),b(X^{i,N}(s),\mu_s^{X,N
})-b_{\delta}(Y^{i,N}(s_{\delta}),\mu_{s_{\delta}}^{Y,N}) \right \rangle \, \mathrm{d}s =\Pi_t^{1,i}+\Pi_t^{2,i}+\Pi_t^{3,i}+\Pi_t^{4,i},
\end{split}
\end{equation*}
where we defined
\begin{align*}
& \Pi_t^{1,i} := \int_0^t \left \langle Z^{i,N}(s),b(X^{i,N}(s),\mu_s^{X,N
})-b(Y^{i,N}(s),\mu_s^{X,N}) \right \rangle \, \mathrm{d}s,\\
& \Pi_t^{2,i} := \int_0^t \left \langle Z^{i,N}(s),b(Y^{i,N}(s),\mu_s^{X,N
})-b(Y^{i,N}(s),\mu_s^{Y,N}) \right \rangle \, \mathrm{d}s,\\
& \Pi_t^{3,i} := \int_0^t \left \langle Z^{i,N}(s),b(Y^{i,N}(s_{\delta}),\mu_{s_{\delta}}^{Y,N})-b_{\delta}(Y^{i,N}(s_{\delta}),\mu_{s_{\delta}}^{Y,N}) \right \rangle \, \mathrm{d}s,\\
& \Pi_t^{4,i} := \int_0^t \left \langle Z^{i,N}(s),b(Y^{i,N}(s),\mu_s^{Y,N})-b(Y^{i,N}(s_{\delta}),\mu_{s_{\delta}}^{Y,N}) \right \rangle \, \mathrm{d}s.
\end{align*}
Employing ({\bf AD}$_b^1$), Lemma \ref{Lem2D} and Lemma \ref{TH:TH3}, shows that
\begin{align*}
\sum_{l=1}^4 \mathbb{E}[\|\Pi^{l,i}\|^{2p}_{\infty,t}] &\leq \varepsilon (\mathbb{E}[\| Z^{i,N} \|^{2p}_{\infty,t}] + {\color{black} \E [\|Z^{i,N}\|^{2p}_{\8,t-\tau}(1+\|X^{i,N}\|^{2pq_2}_{\8,t-\tau} + \|Y^{i,N}\|^{2pq_2}_{\8,t-\tau})] })\\
&+ C \Big \{ \delta^{2p} + \mathbb{E} \int_{0}^{t} | Z^{i,N}(s) |^{2p} \, \mathrm{d}s  
 + \mathbb{E} \int_{0}^{0 \lor (t-\tau)} | Z^{i,N}_s |^{2p} (1+ | X^{i,N}(s) |^{2pq_2} +  | Y^{i,N}(s) |^{2pq_2}) \, \mathrm{d}s \Big \}.  
\end{align*}
The remaining parts of Theorem \ref{THDelay:THDelay4} can be readily shown using above auxiliary results and the techniques employed in the proof of \cite[Theorem 2.1]{stock} combined with an inductive argument (using subintervals of length $\tau$), illustrated in Proposition \ref{Prop:MomentDelayMcKean}.}
\end{proof}

\section{Proof of Theorem \ref{TheoremAntithetic}}
\label{panti}

\subsection{Some auxiliary results}
We start by giving some standard results for stability and one-step error of the processes $Y^{i,N,a}$ and $Y^{i,N,f}$. 
\begin{lem}\label{lemma:differenceTimestep}
Let Assumptions ({\bf AAD}$_b^1$)--({\bf AAD}$_b^3$) and ({\bf AAD}$_\si^1$)--({\bf AAD}$_\si^2$) hold. Let $Y^{i,N,a}$ and $Y^{i,N,f}$ be defined as above. Then, for any $p \geq 2$ there exists a constant $C>0$ such that 
\begin{align*}
& \max_{i \in \mathbb{S}_N}  \max_{n \in \lbrace 0, \ldots, M \rbrace} \mathbb{E} \left[  |Y^{i,N,f}(t_n)|^{p} \right] \leq C, \qquad \max_{i \in  \mathbb{S}_N} \max_{n \in \lbrace 0, \ldots, M-1 \rbrace} \mathbb{E} [|Y^{i,N,f}(t_{n+1/2}) - Y^{i,N,f}(t_n)|^p] \leq C \delta^{p/2}, \\
& \max_{i \in \mathbb{S}_N} \max_{n \in \lbrace 0, \ldots, M \rbrace} \mathbb{E} \left[ |Y^{i,N,a}(t_n)|^{p} \right] \leq C, \qquad \max_{i \in  \mathbb{S}_N} \max_{n \in \lbrace 0, \ldots, M-1 \rbrace} \mathbb{E} [|Y^{i,N,a}(t_{n+1/2}) - Y^{i,N,a}(t_n)|^{p}] \leq C \delta^{p/2}.
\end{align*} 
\begin{proof}
The moment stability of $Y^{i,N,f}$ and $Y^{i,N,a}$ is a consequence of the main results stated in Section \ref{Section:Sec2}. For the one-step errors, we remark that H\"{o}lder's inequality yields
\begin{align*}
& |Y^{i,N,f}({t_{n+1/2}}) - Y^{i,N,f}({t_{n}})|^{p} \\
& \leq C \Big\{  |b_\dd(Y^{i,N,f}(t_n), \mu_{t_n}^{Y,N,f}) \delta/2|^{p} + |\sigma(Y^{i,N,f}(t_n), Y^{i,N,f}(t_n- \tau)) \delta W^{i}_n|^p \nonumber \\
& \qquad + \Big| \sigma(Y^{i,N,f}(t_n), Y^{i,N,f}(t_n- \tau)) \partial_{x_1} \sigma (Y^{i,N,f}(t_n), Y^{i,N,f}(t_n- \tau)) \frac{(\delta W^{i}_n)^2- \delta/2}{2} \Big|^p \nonumber  \\
& \qquad + \Big|\sigma(Y^{i,N,f}(t_n- \tau), Y^{i,N,f}(t_n- 2\tau)) \partial_{x_2} \sigma (Y^{i,N,f}(t_n), Y^{i,N,f}(t_n- \tau))\frac{\delta W^{i}_n \delta W^{i}_{n- \tau}}{2} \Big|^p \Big \}.
\end{align*}
Then, taking expectations on both sides, employing the growth assumptions on the coefficients along with the moment stability of $Y^{i,N,f}$ and standard estimates for Brownian increments allows to deduce the claim.
\end{proof}
\end{lem}

We proceed by stating another standard lemma, which bounds the difference of $Y^{i,N,f}$ and $Y^{i,N,a}$ over a coarse time-step and will be needed throughout this section:
\begin{lem}\label{lemma:antitheticLemma2}
Let Assumptions ({\bf AAD}$_b^1$)--({\bf AAD}$_b^3$) and ({\bf AAD}$_\si^1$)--({\bf AAD}$_\si^2$) hold. Then, for any $p \geq 2$ there exists a constant $C>0$ such that 
\begin{align*}
\max_{i \in \mathbb{S}_N} \max_{n \in \lbrace 0, \ldots, M \rbrace} \mathbb{E} [|Y^{i,N,f}(t_{n}) - Y^{i,N,a}(t_n)|^{p}] \leq C \delta^{p/2}.
\end{align*}
\end{lem}
\begin{proof}
The proof follows analogous steps to \cite[Lemma 4.6]{MGLS} in combination with Lemma \ref{lem_coarse} and is therefore omitted.
\end{proof}

In a next step, we will represent $Y^{i,N,f}$ and $Y^{i,N,a}$, {\color{black}defined by \eqref{antithetic:f} and \eqref{antithetic:a}, respectively}, over a single coarse time-step.
This is necessary in order to give 
an approximation result for $\overline{Y}^{i,N,f}$, see \eqref{Y_bar}, on a coarse grid, which consequently will enable us to derive an estimate for the difference between $\overline{Y}^{i,N,f}$ and $Y^{i,N,c}$, see \eqref{eq:coarsepath}, over a coarse grid in an $L_2$-sense: 
\begin{lem}\label{lemma:AntiLemma1}
{\color{black}
Let Assumptions ({\bf AAD}$_b^1$)--({\bf AAD}$_b^3$) and ({\bf AAD}$_\si^1$)--({\bf AAD}$_\si^2$) hold and let $p \geq 2$. Let $x \in \lbrace f,a \rbrace$ and we use the convention $\mathrm\mathrm{sign}(f): = -$ and $\mathrm{sign}(a):=+$. For all $n \in \lbrace 0, \ldots, M-1 \rbrace$, there exist $N_{i,n,x}$ and $M_{i,n,x}$,
where $\mathbb{E} \left[M_{i,n,x}| \mathcal{F}_{t_n} \right] = 0$ and for some constant $C >0$
\begin{align*}
& \max_{n \in \lbrace 0, \ldots, M-1 \rbrace} \mathbb{E}[|M_{i,n,x}|^p] \leq C \delta^{3p/2}, \quad \max_{n \in \lbrace 0, \ldots, M- 1 \rbrace} \mathbb{E}[|N_{i,n,x}|^p] \leq C \delta^{2p},
\end{align*} 
such that the difference equation for $Y^{i,N,x}$ can be written as 
\begin{align*}
 Y^{i,N,x} ({t_{n+1}}) &= \mathscr{S}\left(Y^{i,N,x}({t_{n}}), Y^{i,N,x}({t_{n}-\tau}), Y^{i,N,x}({t_{n}- 2\tau}), \mu_{t_n}^{Y,N,x}, \delta, \Delta W^{i}_n, \Delta W^{i}_n \Delta W^{i}_{n- \tau} \right) \\ 
& \quad \mathrm{sign}(x) \sigma(Y^{i,N,x}(t_{n}- \tau), Y^{i,N,x}(t_n- 2\tau)) \partial_{x_2} \sigma (Y^{i,N,x}(t_{n}), Y^{i,N,x}(t_n- \tau)) \\
& \qquad \times \frac{1}{2} \left( \delta W^{i}_{n} \delta W^{i}_{n+1/2- \tau} - \delta W^{i}_{n- \tau} \delta W^{i}_{n+1/2} \right) + N_{i,n,x} + M_{i,n,x}.
\end{align*}}
\end{lem}
\begin{proof}
{\color{black}We observe that $N_{i,n,f} + M_{i,n,f}$ can be expressed as the sum of the following terms indexed by $f$ (and analogously for $N_{i,n,a} + M_{i,n,a}$):}
\begin{align*}
R_{i,n,x} &:= (b(Y^{i,N,x}(t_{n+1/2}), \mu_{t_{n+1/2}}^{Y,N,x}) - b(Y^{i,N,x}(t_n), \mu_{t_n}^{Y,N,x})) \frac{\delta}{2} \\
B_{i,n,x} &:=   (b(Y^{i,N,x}(t_n), \mu_{t_n}^{Y,N,x}) - b_\dd(Y^{i,N,x}(t_n), \mu_{t_n}^{Y,N,x})) \frac{\delta}{2} \\ 
& \qquad + (b_\dd(Y^{i,N,x}(t_{n+1/2}), \mu_{t_{n+1/2}}^{Y,N,x})  - b(Y^{i,N,x}(t_{n+1/2}), \mu_{t_{n+1/2}}^{Y,N,x})) \frac{\delta}{2} \\ 
M^{(2)}_{i,n,f} &:=  \Big(\sigma(Y^{i,N,f}(t_{n+1/2}), Y^{i,N,f}(t_{n+1/2}- \tau))  - \sigma(Y^{i,N,f}(t_n), Y^{i,N,f}(t_n- \tau))  \\
& \qquad -  \sigma(Y^{i,N,f}(t_{n}), Y^{i,N,f}(t_n- \tau)) \partial_{x_1} \sigma(Y^{i,N,f}(t_{n}), Y^{i,N,f}(t_n- \tau))\delta W^{i}_{n} \\
& \qquad -  \sigma(Y^{i,N,f}(t_{n}- \tau), Y^{i,N,f}(t_n- 2\tau))  \partial_{x_2} \sigma(Y^{i,N,f}(t_{n}), Y^{i,N,f}(t_n- \tau))\delta W^{i}_{n- \tau} \Big) \delta W^{i}_{n+1/2}  \\
{\color{black} M^{(2)}_{i,n,a}} &:=  \Big(\sigma(Y^{i,N,a}(t_{n+1/2}), Y^{i,N,a}(t_{n+1/2}- \tau))  - \sigma(Y^{i,N,a}(t_n), Y^{i,N,a}(t_n- \tau))  \\
& \qquad -  \sigma(Y^{i,N,a}(t_{n}), Y^{i,N,a}(t_n- \tau)) \partial_{x_1} \sigma(Y^{i,N,a}(t_{n}), Y^{i,N,a}(t_n- \tau))\delta W^{i}_{n +1/2} \\
& \qquad -  \sigma(Y^{i,N,a}(t_{n}- \tau), Y^{i,N,a}(t_n- 2\tau))  \partial_{x_2} \sigma(Y^{i,N,a}(t_{n}), Y^{i,N,a}(t_n- \tau))\delta W^{i}_{n +1/2- \tau} \Big) \delta W^{i}_{n} \\
M^{(3)}_{i,n,f} & :=  \Big(\sigma(Y^{i,N,f}(t_{n+1/2}), Y^{i,N,f}(t_{n+1/2}- \tau)) \partial_{x_1} \sigma(Y^{i,N,f}(t_{n+1/2}), Y^{i,N,f}(t_{n+1/2}- \tau))  \\
& \qquad - \sigma(Y^{i,N,f}(t_{n}), Y^{i,N,f}(t_n- \tau))  \partial_{x_1} \sigma(Y^{i,N,f}(t_{n}), Y^{i,N,f}(t_n- \tau)) \Big)\frac{(\delta W^{i}_{n+1/2})^2 - \delta/2}{2} \\
{\color{black} M^{(3)}_{i,n,a}} & :=  \Big(\sigma(Y^{i,N,a}(t_{n+1/2}), Y^{i,N,a}(t_{n+1/2}- \tau)) \partial_{x_1} \sigma(Y^{i,N,a}(t_{n+1/2}), Y^{i,N,a}(t_{n+1/2}- \tau))  \\
& \qquad - \sigma(Y^{i,N,a}(t_{n}), Y^{i,N,a}(t_n- \tau))  \partial_{x_1} \sigma(Y^{i,N,a}(t_{n}), Y^{i,N,a}(t_n- \tau)) \Big)\frac{(\delta W^{i}_{n})^2 - \delta/2}{2} \\
M^{(4)}_{i,n,f} & :=  \Big(\sigma(Y^{i,N,f}(t_{n+1/2}- \tau), Y^{i,N,f}(t_{n+1/2}- 2\tau))  \partial_{x_2} \sigma (Y^{i,N,f}(t_{n+1/2}), Y^{i,N,f}(t_{n+1/2}- \tau)) \\
& \qquad -\sigma(Y^{i,N,f}(t_{n}- \tau), Y^{i,N,f}(t_n- 2\tau))  \partial_{x_2} \sigma (Y^{i,N,f}(t_{n}), Y^{i,N,f}(t_n- \tau))\Big)  \frac{\delta W^{i}_{n+1/2} \delta W^{i}_{n+1/2- \tau}}{2} \\
{\color{black}M^{(4)}_{i,n,a}} & :=  \Big(\sigma(Y^{i,N,a}(t_{n+1/2}- \tau), Y^{i,N,a}(t_{n+1/2}- 2\tau))  \partial_{x_2} \sigma (Y^{i,N,a}(t_{n+1/2}), Y^{i,N,a}(t_{n+1/2}- \tau)) \\
& \qquad -\sigma(Y^{i,N,a}(t_{n}- \tau), Y^{i,N,a}(t_n- 2\tau))  \partial_{x_2} \sigma (Y^{i,N,a}(t_{n}), Y^{i,N,a}(t_n- \tau))\Big)  \frac{\delta W^{i}_{n} \delta W^{i}_{n- \tau}}{2}.
\end{align*}
{\color{black}In the following, for simplicity, we will restrict the discussion to $x=f$}. We first observe that {\color{black}
\begin{align*}
B_{i,n,f} &= \left( \frac{\delta|b(Y^{i,N,f}(t_n), \mu_{t_n}^{Y,N,f})|^{\bar{q}} b(Y^{i,N,f}(t_n), \mu_{t_n}^{Y,N,f})}{1 + \delta|b(Y^{i,N,f}(t_n), \mu_{t_n}^{Y,N,f})|^{\bar{q}}} \right) \frac{\delta}{2}  \\
& \qquad - \frac{\delta|b(Y^{i,N,f}(t_{n+1/2}), \mu_{t_{n+1/2}}^{Y,N,f})|^{\bar{q}} b(Y^{i,N,f}(t_{n+1/2}), \mu_{t_{n+1/2}}^{Y,N,f})}{1 + \delta|b(Y^{i,N,f}(t_{n+1/2}), \mu_{t_{n+1/2}}^{Y,N,f})|^{\bar{q}}}  \frac{\delta}{2},
\end{align*}
which implies $\max_{\lbrace 0, \ldots, M-1 \rbrace} \mathbb{E}[|B_{i,n,f}|^{p}] \leq C \delta^{2p}$, due to Lemma \ref{lemma:differenceTimestep} and ({\bf AAD}$_b^1$). }

Next, we aim to analyse the term $R_{i,n,f}$ and assume for ease of notation that $b(x,\mu) = b_1(x) + b_2(\mu)$ for any $x \in \R$, $\mu \in \mathcal{P}_2(\R)$. Then, we may write 
\begin{align*}
& b(Y^{i,N,f}(t_{n+1/2}), \mu_{t_{n+1/2}}^{Y,N,f}) - b(Y^{i,N,f}(t_n), \mu_{t_n}^{Y,N,f}) \\
&= b_1(Y^{i,N,f}(t_{n+1/2})) - b_1(Y^{i,N,f}(t_n))+ b_2(\mu_{t_{n+1/2}}^{Y,N,f}) - b_2(\mu_{t_n}^{Y,N,f}), 
\end{align*}
and compute the expansion
\begin{align*}
b_1(Y^{i,N,f}(t_{n+1/2})) - b_1(Y^{i,N,f}(t_n)) &= \partial_x b_1(Y^{i,N,f}(t_n))(Y^{i,N,f}(t_{n+1/2})-Y^{i,N,f}(t_n)) 
\\
& \quad +  \frac{1}{2} \partial^{2}_x b_1(\xi^{i,N}) (Y^{i,N,f}(t_{n+1/2})-Y^{i,N,f}(t_n))^2,
\end{align*}
where $\xi^{i,N}$ lies on the line between $Y^{i,N,f}(t_{n+1/2})$ and $Y^{i,N,f}(t_{n})$. In case the drift is not decomposable, we have to additionally analyse {\color{black}
\begin{align*}
&\partial_x b(Y^{i,N,f}(t_n),\mu_{t_{n+1/2}}^{Y,N,f})(Y^{i,N,f}(t_{n+1/2})-Y^{i,N,f}(t_n))  \\
& = \partial_x b(Y^{i,N,f}(t_n),\mu_{t_{n}}^{Y,N,f})(Y^{i,N,f}(t_{n+1/2})-Y^{i,N,f}(t_n)) \\
&\quad  + \left(\partial_x b(Y^{i,N,f}(t_n),\mu_{t_{n+1/2}}^{Y,N,f}) - \partial_x b(Y^{i,N,f}(t_n),\mu_{t_{n}}^{Y,N,f}) \right)(Y^{i,N,f}(t_{n+1/2})-Y^{i,N,f}(t_n)),    
\end{align*}
which along with ({\bf AAD}$_b^2$) allows to treat this term analogously to $N^{(1)}_{i,n,f}$ precisely introduced below.}

We continue by writing, 
\begin{align*}
(b_1(Y^{i,N,f}(t_{n+1/2}) - b_1(Y^{i,N,f}(t_n))) \frac{\delta}{2} = M_{i,n,f}^{(1)} + N^{(1)}_{i,n,f},
\end{align*}
where we introduced
\begin{align*}
M_{i,n,f}^{(1)} &:= \partial_x b_1(Y^{i,N,f}(t_n)) \sigma(Y^{i,N}(t_n), Y^{i,N}(t_n- \tau)) \delta W^{i}_{n} \frac{\delta}{2}  \\
N^{(1)}_{i,n,f} &:= \partial_x b_1(Y^{i,N,f}(t_n)) \Big( b_\dd(Y^{i,N,f}(t_n), \mu_{t_n}^{Y,N,f}) \frac{\delta}{2} \\
& \quad + \sigma(Y^{i,N,f}(t_n), Y^{i,N,f}(t_n- \tau)) \partial_{x_1} \sigma(Y^{i,N,f}(t_n), Y^{i,N,f}(t_n- \tau)) \frac{(\delta W^{i}_{n})^2- \delta/2}{2} \\
& \quad + \sigma(Y^{i,N,f}(t_n-\tau), Y^{i,N,f}(t_n- 2\tau)) \partial_{x_2} \sigma(Y^{i,N,f}(t_n), Y^{i,N,f}(t_n- \tau)) \frac{\delta W^{i}_{n} \delta W^{i}_{n- \tau}}{2} \Big) \frac{\delta}{2} \\
& \quad + \frac{1}{2} \partial^{2}_{x}b_1(\xi^{i,N}) (Y^{i,N,f}(t_{n+1/2})-Y^{i,N,f}(t_n))^2 \frac{\delta}{2}. 
\end{align*}
Using the notation {\color{black}$\tilde{\mu}_{t_{n}}^{Y,N,f,\lambda}(\mathrm{d}x) := \frac{1}{N} \sum_{j=1}^{N} \delta_{\lambda Y^{i,N,f}(t_{n + 1/2}) + (1-\lambda)Y^{i,N,f}(t_n)}(\mathrm{d}x)$}, where $\lambda \in [0,1]$, and 
\begin{align*}
\Delta Y^{i,N}(t_{n+1/2}) :&= Y^{i,N,f}(t_{n+1/2}) - Y^{i,N,f}(t_n)\\
& = b_\dd(Y^{i,N,f}(t_n), \mu_{t_n}^{Y,N,f}) \frac{\delta}{2} \\
& \quad + \sigma(Y^{i,N,f}(t_n), Y^{i,N,f}(t_n- \tau)) \delta W^{i}_{n}  \\
& \quad + \sigma(Y^{i,N,f}(t_n), Y^{i,N,f}(t_n- \tau)) \partial_{x_1} \sigma(Y^{i,N,f}(t_n), Y^{i,N,f}(t_n- \tau)) \frac{(\delta W^{i}_{n})^2- \delta/2}{2} \\
& \quad + \sigma(Y^{i,N,f}(t_n-\tau), Y^{i,N,f}(t_n-2\tau))\partial_{x_2} \sigma(Y^{i,N,f}(t_n), Y^{i,N,f}(t_n- \tau))\frac{\delta W^{i}_{n} \delta W^{i}_{n- \tau}}{2},
\end{align*}
we obtain, as $\mathcal{P}_2(\R) \ni \mu \mapsto b(x,\mu)$ is continuously $L$-differentiable  
\begin{align*}
&\left( b_2(\mu_{t_{n+1/2}}^{Y,N,f}) - b_2(\mu_{t_{n}}^{Y,N,f}) \right)\frac{\delta}{2} \\
& = \int_{0}^{1}  \frac{\d}{\d \lambda} b_2(\tilde{\mu}_{t_{n}}^{Y,N,f,\lambda})  \, \mathrm{d}\lambda \ \frac{\delta}{2}\\
& = \int_{0}^{1} \frac{\delta}{2N} \sum_{j=1}^{N} \partial_{\mu}b_2(\tilde{\mu}_{t_{n}}^{Y,N,f,\lambda})(\lambda Y^{j,N,f}(t_{n + 1/2}) + (1-\lambda)Y^{j,N,f}(t_n))(Y^{j,N,f}(t_{n + 1/2})-Y^{j,N,f}(t_{n}))  \, \mathrm{d}\lambda  \\
& =  \int_{0}^{1} \frac{\delta}{2N}\sum_{j=1}^{N} \partial_{\mu}b_2(\tilde{\mu}_{t_{n}}^{Y,N,f,\lambda}) (Y^{j,N,f}(t_{n}) + \lambda \Delta Y^{j,N}(t_{n+1/2})  ) \Delta Y^{j,N}(t_{n+1/2})  \, \mathrm{d}\lambda \\
& =  \frac{\delta}{2N} \sum_{j=1}^{N} \partial_{\mu}b_2(\mu_{t_{n}}^{Y,N,f}) (Y^{i,N,f}(t_{n})) \Delta Y^{j,N}(t_{n+1/2})  \\
& \quad +   \frac{\delta}{2N} \sum_{j=1}^{N} \int_{0}^{1}  \left(\partial_{\mu}b_2(\tilde{\mu}_{t_{n}}^{Y,N,f,\lambda}) (Y^{j,N,f}(t_{n}) + \lambda \Delta Y^{j,N}(t_{n+1/2}))  - \partial_{\mu}b_2(\mu_{t_{n}}^{Y,N,f}) (Y^{j,N,f}(t_{n})) \right) \Delta Y^{j,N}(t_{n+1/2})   \, \mathrm{d}\lambda   \\
&=: R_{i,n,f}^{(1)} + R_{i,n,f}^{(2)}.
\end{align*}
Note that $R_{i,n,f}^{(1)}$ can be decomposed as
\begin{equation*}
R_{i,n,f}^{(1)} = R_{i,n,f}^{(1,1)} + R_{i,n,f}^{(1,2)},
\end{equation*}
where 
\begin{align*}
& \mathbb{E}\left[ R_{i,n,f}^{(1,1)} | \mathcal{F}_{t_n} \right] = 0, \quad \max_{n \in \lbrace 0, \ldots, M-1 \rbrace} \mathbb{E}[|R_{i,n,f}^{(1,1)}|^p] \leq C \delta^{3p/2}, \quad  \max_{n \in \lbrace 0, \ldots, M-1 \rbrace} \mathbb{E}[|R_{i,n,f}^{(1,2)}|^p] \leq C \delta^{2p}.
\end{align*}
{\color{black}The terms $N_{i,n,x}$ and $M_{i,n,x}$ introduced in Lemma \ref{lemma:AntiLemma1} can be written as 
\begin{align*}
N_{i,n,x} &= N^{(1)}_{i,n,x} + B_{i,n,x} + R^{(2)}_{i,n,x} + R^{(1,2)}_{i,n,x} \\
M_{i,n,x} &= M^{(1)}_{i,n,x} + M^{(2)}_{i,n,x} + M^{(3)}_{i,n,x}+ M^{(4)}_{i,n,x} + R^{(1,1)}_{i,n,x}.
\end{align*}
We continue to show the claimed moment estimates for $R_{i,n,f}^{(2)}$ and the martingale terms.} Due to assumption ({\bf AAD}$_b^2$), we obtain
\begin{align*}
|R_{i,n,f}^{(2)}| & \leq  \frac{\delta}{2N} \sum_{j=1}^{N} \int_{0}^{1} \left| \partial_{\mu}b_2(\tilde{\mu}_{t_{n}}^{Y,N,f,\lambda}) (Y^{j,N,f}(t_{n}) + \lambda \Delta Y^{j,N}(t_{n+1/2}))  - \partial_{\mu}b_2(\mu_{t_{n}}^{Y,N,f}) (Y^{j,N,f}(t_{n})) \right| | \Delta Y^{j,N}(t_{n+1/2})|  \, \mathrm{d}\lambda  \\
& \leq \frac{C\delta}{2N} \sum_{j=1}^{N} \int_{0}^{1} | \mathcal{W}_2(\tilde{\mu}_{t_{n}}^{Y,N,f,\lambda},\mu_{t_{n}}^{Y,N,f}) {\color{black}+} \lambda \Delta Y^{j,N}(t_{n+1/2})| |\Delta Y^{j,N}(t_{n+1/2})|   \, \mathrm{d}\lambda \\
& \leq  \frac{C\delta}{2N} \sum_{j=1}^{N} \int_{0}^{1} \left| \left(\frac{1}{N} \sum_{l=1}^{N} |Y^{l,N,f}(t_{n + 1/2})-Y^{l,N,f}(t_{n})|^2 \right)^{1/2} {\color{black}+} \lambda \Delta Y^{j,N}(t_{n+1/2}) \right| |\Delta Y^{j,N}(t_{n+1/2})|   \, \mathrm{d}\lambda.
\end{align*}
Using the fact that all particles are identically distributed, Lemma \ref{lemma:differenceTimestep} gives, for $p \geq 2$,  
\begin{equation*}
  \max_{n \in \lbrace 0, \ldots, M-1 \rbrace} \mathbb{E}[|R_{i,n,f}^{(2)}|^p] \leq C \delta^{2p}.
\end{equation*}
We continue with estimating the martingale terms. The above expression for $M_{i,n,f}^{(2)}$ can be rewritten as
{\color{black}
\begin{align*}
M_{i,n,f}^{(2)} &= \partial_{x_1} \sigma(Y^{i,N,f}(t_n), Y^{i,N,f}(t_n- \tau)) Q^{i,f}(t_n) \delta W^{i}_{n+1/2} \\
& \quad + \partial_{x_2} \sigma(Y^{i,N,f}(t_n), Y^{i,N,f}(t_n- \tau)) Q^{i,f}(t_n- \tau) \delta W^{i}_{n+1/2} \\
& \quad + \left(\frac{1}{2} \sum_{j,k=1}^{2} \partial^2_{x_j x_k} \sigma(\xi_2^{i,N}, \xi_3^{i,N})  \Delta Y^{i,f}(t_{n+1/2}-(j-1)\tau)  \Delta Y^{i,f}(t_{n+1/2}-(k-1)\tau) \right) \delta W^{i}_{n+1/2},
\end{align*}
where 
\begin{align*}
Q^{i,f}(t_n) & := b_\dd(Y^{i,N,f}(t_n), \mu_{t_n}^{Y,N,f}) \frac{\delta}{2} \\
& \quad + \sigma(Y^{i,N,f}(t_n), Y^{i,N,f}(t_n- \tau)) \partial_{x_1} \sigma(Y^{i,N,f}(t_n), Y^{i,N,f}(t_n- \tau)) \frac{(\delta W^{i}_{n})^2- \delta/2}{2} \\
& \quad + \sigma(Y^{i,N,f}(t_n-\tau), Y^{i,N,f}(t_n-2\tau))\partial_{x_2} \sigma(Y^{i,N,f}(t_n), Y^{i,N,f}(t_n- \tau))\frac{\delta W^{i}_{n} \delta W^{i}_{n- \tau}}{2},
\end{align*}
}
and $\xi_2^{i,N}, \xi_3^{i,N}$ on the line between $Y^{i,N,f}(t_n)$ and $Y^{i,N,f}(t_{n+1/2})$ and $Y^{i,N,f}(t_n- \tau)$ and $Y^{i,N,f}(t_{n+1/2}- \tau)$, respectively.

Further, considering $M_{i,n,f}^{(3)}$ and $M_{i,n,f}^{(4)}$, we have using a first order Taylor series expansion {\color{black}
\begin{align*}
M_{i,n,f}^{(3)} = \left(\sum_{j=1}^{2} \partial_{x_j} \bar{\sigma}(\xi_4^{i,N},\xi_5^{i,N})   \Delta Y^{i,f}(t_{n+1/2}-(j-1)\tau) \right) \frac{(\delta W^{i}_{n+1/2})^2 - \delta/2}{2},
\end{align*}
where 
\begin{align*}
\bar{\sigma}(x,y) = \sigma(x,y) \partial_x \sigma(x,y),
\end{align*}
and $\xi_4^{i,N}, \xi_5^{i,N}$ on the line between $Y^{i,N,f}(t_n)$ and $Y^{i,N,f}(t_{n+1/2})$ and $Y^{i,N,f}(t_n- \tau)$ and $Y^{i,N,f}(t_{n+1/2}- \tau)$, respectively. Additionally, we have 
\begin{align*}
M_{i,n,f}^{(4)} = \left(\sum_{j=1}^{2} \partial_{x_j} \hat{\sigma}(\xi_6^{i,N},\xi_7^{i,N})   \Delta Y^{i,f}(t_{n+1/2}-(j-1)\tau) \right) \frac{\delta W^{i}_{n+1/2} \delta W^{i}_{n+1/2- \tau}}{2}, 
\end{align*}}
where 
\begin{align*}
\hat{\sigma}(x,y) = \sigma(x,y) \partial_y \sigma(x,y),
\end{align*}
and $\xi_6^{i,N}, \xi_7^{i,N}$ on the line between $Y^{i,N,f}(t_n- \tau)$ and $Y^{i,N,f}(t_{n+1/2}- \tau)$ and $Y^{i,N,f}(t_n- 2\tau)$ and $Y^{i,N,f}(t_{n+1/2}- 2\tau)$, respectively. From here the claimed estimates for the moments of $M_{i,n,f}^{(2)}$, $M_{i,n,f}^{(3)}$ and $M_{i,n,f}^{(4)}$ follow.
\end{proof}

We have  
the following approximation of the 
antithetic scheme: 
\begin{lem}\label{AntitheticLemma2}
{\color{black}
Let Assumptions ({\bf AAD}$_b^1$)--({\bf AAD}$_b^3$) and ({\bf AAD}$_\si^1$)--({\bf AAD}$_\si^2$) hold and let $p \geq 2$. Then for each $n \in \lbrace 0,\ldots, M-1 \rbrace$ there exist $\tilde{M}_{i,n}$
with $\mathbb{E} \left[\tilde{M}_{i,n}| \mathcal{F}_{t_n} \right] = 0$ and $p$-th moments of order $\delta^{3p/2}$,
and $\tilde{B}_{i,n}$ and $\tilde{N}_{i,n}$ with $p$-th moments of order $\delta^{2p}$, such that the difference equation for $\overline{Y}^{i,N,f}$ reads
\begin{align*}
\overline{Y}^{i,N,f}(t_{n+1}) &= \mathscr{S}\left(\overline{Y}^{i,N,f}({t_{n}}), \overline{Y}^{i,N,f}({t_{n}-\tau}), \overline{Y}^{i,N,f}({t_{n}- 2\tau}), \mu_{t_n}^{\overline{Y},N}, \delta, \Delta W^{i}_n, \Delta W^{i}_n \Delta W^{i}_{n- \tau} \right) \\
& \quad + \frac{1}{2}(N_{i,n,f} + N_{i,n,a} + M_{i,n,f} + M_{i,n,a})  + \tilde{B}_{i,n} + \tilde{N}_{i,n} + \tilde{M}_{i,n},
\end{align*}
where $N_{i,n,x}$ and $M_{i,n,x}$ are defined as in Lemma \ref{lemma:AntiLemma1} and we set
\begin{align*}
\mu_{t_{n}}^{\overline{Y},N}(\mathrm{d}x):= \frac{1}{N} \sum_{j=1}^{N} \delta_{\overline{Y}^{j,N,f}(t_{n})}(\mathrm{d}x).
\end{align*}}
\end{lem}

%
%
%
%
%
\begin{proof}
In Lemma \ref{AntitheticLemma2} we introduced the abbreviations
\begin{align*}
\tilde{N}_{i,n} &:= \left( \frac{1}{2} \left( b(Y^{i,N,f}(t_{n}), \mu_{t_{n}}^{Y,N,f}) + b(Y^{i,N,a}(t_{n}), \mu_{t_{n}}^{Y,N,a})  \right) - b(\overline{Y}^{i,N,f}(t_{n}), \mu_{t_{n}}^{\overline{Y},N}) \right) \delta \\
\tilde{B}_{i,n} & := \left( \frac{1}{2} \left( (b_\dd- b)(Y^{i,N,f}(t_{n}), \mu_{t_{n}}^{Y,N,f}) + (b_\dd-b)(Y^{i,N,a}(t_{n}), \mu_{t_{n}}^{Y,N,a})  \right) - (b-b_\dd)(\overline{Y}^{i,N,f}(t_{n}), \mu_{t_{n}}^{\overline{Y},N}) \right) \delta,
\end{align*}
and $\tilde{M}_{i,n}$ can be decomposed in the form
\begin{align*}
\tilde{M}_{i,n}^{(1)} &:= \Big( \frac{1}{2} \left( \sigma(Y^{i,N,f}(t_n), Y^{i,N,f}(t_n- \tau)) + \sigma(Y^{i,N,a}(t_n), Y^{i,N,a}(t_n- \tau)) \right) \\ 
& \qquad - \sigma(\overline{Y}^{i,N,f}(t_n), \overline{Y}^{i,N,f}(t_n- \tau)) \Big) \Delta W^{i}_n \\  
\tilde{M}_{i,n}^{(2)} &:= \Bigg( \frac{1}{2} \Big( \sigma(Y^{i,N,f}(t_{n}), Y^{i,N,f}(t_n- \tau)) \partial_{x_1} \sigma(Y^{i,N,f}(t_{n}), Y^{i,N,f}(t_n- \tau)) \\
& \qquad +\sigma(Y^{i,N,a}(t_{n}), Y^{i,N,a}(t_n- \tau)) \partial_{x_1} \sigma(Y^{i,N,a}(t_{n}), Y^{i,N,a}(t_n- \tau)) \Big) \\
& \qquad  - \sigma(\overline{Y}^{i,N,f}(t_n), \overline{Y}^{i,N,f}(t_n- \tau))  \partial_{x_1} \sigma(\overline{Y}^{i,N,f}(t_n), \overline{Y}^{i,N,f}(t_n- \tau)) \Bigg)\frac{(\Delta W^{i}_n)^2- \delta}{2}   \\
\tilde{M}_{i,n}^{(3)} &:= \Bigg( \frac{1}{2} \Big( \sigma(Y^{i,N,f}(t_{n}-\tau), Y^{i,N,f}(t_n- 2\tau)) \partial_{x_2} \sigma(Y^{i,N,f}(t_{n}), Y^{i,N,f}(t_n- \tau)) \\
& \qquad +\sigma(Y^{i,N,a}(t_{n}-\tau), Y^{i,N,a}(t_n- 2\tau)) \partial_{x_2} \sigma(Y^{i,N,a}(t_{n}), Y^{i,N,a}(t_n- \tau)) \Big) \\
& \qquad  - \sigma(\overline{Y}^{i,N,a}(t_n-\tau), \overline{Y}^{i,N,a}(t_n- 2\tau)) \partial_{x_2} \sigma(\overline{Y}^{i,N,a}(t_n), \overline{Y}^{i,N,a}(t_n- \tau)) \Bigg)  \frac{\Delta W^{i}_n \Delta W^{i}_{n- \tau}}{2} \\
\tilde{M}_{i,n}^{(4)} &:= \Big( \sigma(Y^{i,N,a}(t_{n}- \tau), Y^{i,N,a}(t_n- 2\tau)) \partial_{x_2}\sigma(Y^{i,N,a}(t_{n}), Y^{i,N,a}(t_n- \tau)) \\
& \qquad -\sigma(Y^{i,N,f}(t_{n}- \tau), Y^{i,N,f}(t_n- 2\tau))\partial_{x_2} \sigma(Y^{i,N,f}(t_{n}), Y^{i,N,f}(t_n- \tau))\Big)  \\
& \qquad \times \frac{1}{2}(\delta W^{i}_{n} \delta W^{i}_{n+1/2- \tau} - \delta W^{i}_{n- \tau} \delta W^{i}_{n+1/2} ).
\end{align*}
In what follows, we analyse the term $\tilde{N}_{i,n}$ {\color{black}assuming for ease of notation the drift term to be decomposable as in the previous Lemma (the assertion for $\tilde{B}_{i,n}$ follows employing similar arguments to the ones used for the term $B_{i,n,x}$ in Lemma \ref{lemma:AntiLemma1})}. We define, for $\lambda \in [0,1]$ and $n \in \lbrace 0, \ldots, M \rbrace$
{\color{black}
\begin{align*}
m_{t_n}^{f,\lambda}(\mathrm{d}x) :&= \frac{1}{N} \sum_{j=1}^{N} \delta_{(1-\lambda)\overline{Y}^{j,N,f}(t_n) + \lambda Y^{j,N,f}(t_n)}(\mathrm{d}x), \\
m_{t_n}^{a,\lambda}(\mathrm{d}x) :&= \frac{1}{N} \sum_{j=1}^{N} \delta_{(1-\lambda)\overline{Y}^{j,N,f}(t_n) + \lambda Y^{j,N,a}(t_n)}(\mathrm{d}x),  
\end{align*}
and observe that for $x \in \lbrace f,a \rbrace$
\begin{align*}
&b_2(\mu_{t_n}^{\overline{Y},N} + (\mu_{t_n}^{Y,N,x}-\mu_{t_n}^{\overline{Y},N})) - b_2(\mu_{t_n}^{\overline{Y},N}) \\
& = \int_{0}^{1} \frac{\d}{\d \lambda}  b_2(m_{t_n}^{x,\lambda})  \, \mathrm{d}\lambda \ \\
& = \int_{0}^{1} \frac{1}{N} \sum_{j=1}^{N} \partial_{\mu}b_2(m_{t_n}^{x,\lambda})(\lambda Y^{j,N,x}(t_{n}) + (1-\lambda)\overline{Y}^{j,N,f}(t_n))(Y^{j,N,x}(t_{n})-\overline{Y}^{j,N,f}(t_{n}))  \, \mathrm{d}\lambda \\
& = \frac{1}{N} \sum_{j=1}^{N} \partial_{\mu}b_2(\mu_{t_n}^{\overline{Y},N}) (\overline{Y}^{j,N,f}(t_{n})) (Y^{j,N,x}(t_{n})-\overline{Y}^{j,N,f}(t_{n})) \\
& \quad +   \frac{1}{N} \sum_{j=1}^{N} \int_{0}^{1}  \Big(\partial_{\mu}b_2(m_{t_n}^{x,\lambda}) (\overline{Y}^{j,N,f}(t_{n}) + \lambda (Y^{j,N,x}(t_{n})-\overline{Y}^{j,N,f}(t_{n}))) \\
& \qquad  - \partial_{\mu}b_2(\mu_{t_n}^{\overline{Y},N}) (\overline{Y}^{j,N,f}(t_{n})) \Big)(Y^{j,N,x}(t_{n})-\overline{Y}^{j,N,f}(t_{n}))  \, \mathrm{d} \lambda. 
\end{align*}
}
{\color{black}Therefore, performing similar estimates as for $R^{(2)}_{i,n,f}$ in Lemma \ref{lemma:AntiLemma1} and Lemma \ref{lemma:antitheticLemma2} the moments of 
\begin{equation*}
\left( \frac{1}{2} \left( b_2(\mu_{t_{n}}^{Y,N,f}) + b_2(\mu_{t_{n}}^{Y,N,a})  \right)  -  b_2(\mu_{t_{n}}^{\overline{Y},N})  \right) \delta,
\end{equation*}
achieve the asserted order.}
A second order Taylor series expansion in the state component further gives 
\begin{align*}
&\left( \frac{1}{2} \left( b_1(Y^{i,N,f}(t_{n})) + b_1(Y^{i,N,a}(t_{n}))  \right) - b_1(\overline{Y}^{i,N,f}(t_{n})) \right) \delta \\
& \quad = \frac{1}{16} \left( \partial^2_{x} b_1(\xi_8^{i,N}) + \partial^2_x b_1(\xi_9^{i,N}) \right) (Y^{j,N,f}(t_{n}) - Y^{j,N,a}(t_{n}))^2 \delta,
\end{align*}
{\color{black}where $\xi_8^{i,N}$ and $\xi_9^{i,N}$ are on the line between $\overline{Y}^{i,N,f}(t_{n})$ and $Y^{i,N,f}(t_{n})$ and $\overline{Y}^{i,N,f}(t_{n})$ and $Y^{i,N,a}(t_{n})$, respectively. This allows to deduce the claim for $ \tilde{N}_{i,n}$ taking Lemma \ref{lemma:antitheticLemma2} into account}. Here, we only remark that  $\tilde{M}_{i,n}^{(1)}$ can be treated using a second order Taylor series expansions about $(\overline{Y}^{i,N,f}(t_n), \overline{Y}^{i,N,f}(t_n- \tau))$ and $ \tilde{M}_{i,n}^{(2)}, \tilde{M}_{i,n}^{(3)}$ by employing first order Taylor series expansions about $(\overline{Y}^{i,N,f}(t_n), \overline{Y}^{i,N,f}(t_n- \tau))$ or $(\overline{Y}^{i,N,f}(t_n- \tau), \overline{Y}^{i,N,f}(t_n- 2\tau))$, respectively. To further analyse $\tilde{M}_{i,n}^{(4)}$, one can also employ a Taylor series expansion argument and Lemma \ref{lemma:antitheticLemma2}. Note that all of these expressions are martingales and satisfy $\mathbb{E} [|\tilde{M}_{i,n}^{(k)}|^{p}] = \delta^{3p/2}$, for $k \in \lbrace 1, \ldots, 4 \rbrace$.
\end{proof}

\subsection{Proof of Theorem \ref{TheoremAntithetic}}
\begin{proof}
We analyse now the difference of the antithetic approximation and the coarse-path approximation, i.e., we obtain
\begin{align*}
&\mathbb{E}[|\overline{Y}^{i,N,f}(t_{n+1}) - Y^{i,N,c}(t_{n+1})|^2]  = \mathbb{E}[| \overline{Y}^{i,N,f}(t_{n}) - Y^{i,N,c}(t_{n}) \\
& \quad  + (b_\dd(\overline{Y}^{i,N,f}(t_{n})), \mu_{t_{n}}^{\overline{Y},N}) - b_\dd(Y^{i,N,c}(t_{n}), \mu_{t_{n}}^{Y,N,c})) \delta \\
& \quad + (\sigma(\overline{Y}^{i,N,f}(t_n), \overline{Y}^{i,N,f}(t_n- \tau)) - \sigma(Y^{i,N,c}(t_n), Y^{i,N,c}(t_n- \tau))) \Delta W^{i}_{n} \\
& \quad + \sigma(\overline{Y}^{i,N,f}(t_n), \overline{Y}^{i,N,f}(t_n- \tau)) \partial_{x_1} \sigma(\overline{Y}^{i,N,f}(t_n), \overline{Y}^{i,N,f}(t_n- \tau)) \frac{(\Delta W^{i}_{n})^2- \delta}{2} \\
& \quad - \sigma(Y^{i,N,c}(t_n), Y^{i,N,c}(t_n- \tau)) \partial_{x_1} \sigma(Y^{i,N,c}(t_n), Y^{i,N,c}(t_n- \tau)) \frac{(\Delta W^{i}_{n})^2- \delta}{2} \\
& \quad + \sigma(\overline{Y}^{i,N,f}(t_n-\tau), \overline{Y}^{i,N,f}(t_n-2\tau)) \partial_{x_2} \sigma(\overline{Y}^{i,N,f}(t_n), \overline{Y}^{i,N,f}(t_n- \tau))\frac{\Delta W^{i}_{n} \Delta W^{i}_{n- \tau}}{2} \\
& \quad - \sigma(Y^{i,N,c}(t_n-\tau), Y^{i,N,c}(t_n- 2\tau)) \partial_{x_2} \sigma(Y^{i,N,c}(t_n), Y^{i,N,c}(t_n- \tau)) \frac{\Delta W^{i}_{n} \Delta W^{i}_{n- \tau}}{2} \\
& \quad + \frac{1}{2}(N_{i,n,f} + N_{i,n,a} + M_{i,n,f} + M_{i,n,a})  + \tilde{B}_{i,n} + \tilde{N}_{i,n} + \tilde{M}_{i,n}|^2] \\
& \leq (1 + C \delta) \mathbb{E}[|\overline{Y}^{i,N,f}(t_{n}) - Y^{i,N,c}(t_{n})|^2] + C\delta \mathbb{E}[|\overline{Y}^{i,N,f}(t_{n}- \tau) - Y^{i,N,c}(t_{n}- \tau)|^2] + C\delta^3,
\end{align*}
where the above square was explicitly computed and each term was estimated, using either a martingale property or standard techniques in combination with the auxiliary Lemma \ref{lemma:antitheticLemma2}.

Iterating above recursion, yields
\begin{align*}
&\mathbb{E}[|\overline{Y}^{i,N,f}(t_{n+1}) - Y^{i,N,c}(t_{n+1})|^2] \1 \delta^2 + \delta \sum_{l=0}^{n}\mathbb{E}[|\overline{Y}^{i,N,f}(t_{l}- \tau) - Y^{i,N,c}(t_{l}- \tau)|^2].
\end{align*}
Considering above estimate on the interval $[0,\tau]$, allows us to deduce the claim as 
\begin{equation*}
\sum_{l=0}^{n}\mathbb{E}[|\overline{Y}^{i,N,f}(t_{l}- \tau) - Y^{i,N,c}(t_{l}- \tau)|^2]
\end{equation*}
vanishes on this subinterval. In a next step, we investigate the subinterval $[\tau, 2\tau]$. Here, we can employ the fact that $\mathbb{E}[|\overline{Y}^{i,N,f}(t_{l}- \tau) - Y^{i,N,c}(t_{l}- \tau)|^2] \1 \delta^2$, and hence the assertion follows in this case as well. An inductive argument can then be used up to the final time $T$. 
\end{proof}

\end{document}